\documentclass[11pt]{amsart}

\usepackage{amssymb, amscd, amsthm, mathrsfs}

\newtheorem{Thm}{Theorem}[section]
\newtheorem{Prop}[Thm]{Proposition}
\newtheorem{Cor}[Thm]{Corollary}
\newtheorem{Lem}[Thm]{Lemma}

\theoremstyle{remark}

\newtheorem{Ex}[Thm]{Example}
\newtheorem*{Notation}{Notation}
\newtheorem*{Ack}{Acknowledgments}

\numberwithin{equation}{section}

\newcommand{\Order}{\mathcal{O}}
\newcommand{\into}{\hookrightarrow}
\newcommand{\onto}{\twoheadrightarrow}
\newcommand{\isomto}{\overset{\sim}{\to}}
\newcommand{\isomfrom}{\overset{\sim}{\leftarrow}}

\newcommand{\tensor}{\mathbin{\otimes}}
\newcommand{\closure}[1]{\overline{#1}}

\newcommand{\Z}{\mathbb{Z}}
\newcommand{\Q}{\mathbb{Q}}
\newcommand{\R}{\mathbb{R}}

\newcommand{\F}{\mathbb{F}}
\newcommand{\et}{\mathrm{et}}

\newcommand{\id}{\mathrm{id}}
\newcommand{\Gm}{\mathbf{G}_{m}}

\newcommand{\tor}{\mathrm{tor}}

\newcommand{\dirlim}{\varinjlim}

\newcommand{\sep}{\mathrm{sep}}
\newcommand{\ur}{\mathrm{ur}}
\newcommand{\Mod}{\mathrm{Mod}}

\newcommand{\ideal}[1]{\mathfrak{#1}}

\newcommand{\Adele}{\mathbb{A}}

\newcommand{\cl}{\mathrm{cl}}
\DeclareMathOperator{\Gal}{Gal}
\DeclareMathOperator{\Hom}{Hom}
\DeclareMathOperator{\Aut}{Aut}
\DeclareMathOperator{\Ker}{Ker}
\DeclareMathOperator{\Coker}{Coker}

\DeclareMathOperator{\Spec}{Spec}
\DeclareMathOperator{\Ab}{Ab}

\DeclareMathOperator{\Pic}{Pic}

\DeclareMathOperator{\Lie}{Lie}
\DeclareMathOperator{\Disc}{Disc}
\DeclareMathOperator{\rank}{rk}

\DeclareMathOperator{\sheafhom}{\mathscr{H}\mkern-5mu\mathit{om}}

\DeclareMathOperator{\Cl}{Cl}
\DeclareMathOperator{\length}{length}

\DeclareFontFamily{U}{wncy}{}
\DeclareFontShape{U}{wncy}{m}{n}{<->wncyr10}{}
\DeclareSymbolFont{mcy}{U}{wncy}{m}{n}
\DeclareMathSymbol{\Sha}{\mathord}{mcy}{"58}

\title[One-motives over function fields]
	{Special values of $L$-functions of one-motives over function
	fields}
\author{Thomas H. Geisser}
\address{
	Rikkyo University, Ikebukuro, Tokyo, Japan
}
\email{geisser@rikkyo.ac.jp}
\author{Takashi Suzuki}
\address{
	Chuo University, Kasuga, Tokyo, Japan
}
\email{tsuzuki@gug.math.chuo-u.ac.jp}
\thanks{
The first author is supported by JSPS Grant-in-Aid 18K03258.
	The second author is partially supported by JSPS Grant-in-Aid  18J00415.}
\date{October 16, 2022}

\begin{document}

\begin{abstract}
The purpose of this paper is to give a formula 
for the leading coefficient at $s=1$ of the $L$-function of one-motives 
over function fields in terms of Weil-\'etale cohomology, 
generalizing the Weil-\'etale version of the Birch and Swinnerton-Dyer
conjecture in the authors' previous work. As a consequence we express 
the Tamagawa number of a torus introduced by Ono-Oesterl\'e in terms
of Weil-\'etale cohomology, and reprove their Tamagawa number formula.
\end{abstract}

\maketitle


\section{Introduction}
\label{sec: Introduction}

Let $S$ be a proper smooth geometrically connected curve
over a finite field $\F_{q}$ with function field $K$,
and let $M= [X \to G]$ be a $1$-motive over $K$, that is,
a lattice $X$ and a semi-abelian variety $G$ over $K$
placed in degrees $-1$ and $0$, respectively.
Consider  the Hasse-Weil 
$L$-function $L(M, s)$ of the $l$-adic representation $V_{l}(M)(-1)$ over $K$
	\[
			L(M, s)
		=
			\prod_{v}
			\det \bigl(
				1 - \varphi_{v} N(v)^{- s} \,|\, V_{l}(M)(-1)^{I_{v}}
			\bigr)^{-1},
	\]
where $l$ is a prime different from the characteristic $p$ of $K$,  
$v$ runs through the places of $S$, $I_{v}$ is the inertia group at $v$, 
$\varphi_{v}$ is the geometric Frobenius at $v$, and
and $N(v)$ the order of the residue field $k(v)$ at $v$.
Denote the N\'eron model and the connected N\'eron model of $G$ over $S$ by
$\mathcal{G}$ and $\mathcal{G}^{0}$, respectively,
and let $\Lie \mathcal{G}^{0}$ be the Lie algebra of $\mathcal{G}^{0}$
(a locally free sheaf).
Let $\mathcal{X} = j_{\ast} X$, where $j \colon \Spec K \into S$ is the inclusion.
Define $\mathcal{X}^{\Delta}$ by the fiber product
	\[
		\begin{CD}
				\mathcal{X}^{\Delta} @>>> \mathcal{G}^{0}
				\\ @VVV @VVV \\
				\mathcal{X} @>>> \mathcal{G},
		\end{CD}
	\]
and let $\mathcal{M}^{\Delta}$ be the complex of \'etale sheaves 
$[\mathcal{X}^{\Delta} \to \mathcal{G}^{0}]$ on $S$. If we set
\begin{equation}\label{rm}
 r_{M} :=
-\sum_{i}(-1)^{i}\cdot i \cdot\dim H^{i}_W(S, \mathcal{M}^{\Delta}) \tensor_{\Z} \Q,
\end{equation}
then 
$$ r_M= \rank A(K) - \rank X(K) - \rank Y(K)
+ \sum_{v}\rank (\mathcal{X} / \mathcal{X}^{\Delta})(k(v)),$$
where $A$ is the abelian variety quotient of $G$ and
 $Y$ is the character module of the torus part of $G$.
Note that $r_{M}$ can be negative. We give the following formula
in the spirit of Lichtenbaum \cite{Lic09}.

\begin{Thm} \label{mainthm}
Assume that the Tate-Shafarevich group $\Sha(A)$ of $A$ is finite. 
Then the groups 
$H_{W}^{\ast}(S, \mathcal{M}^{\Delta})$ are finitely generated,
$L(M, s)$ has a zero of order $r_M$ at $s=1$, and  
	\[
			\lim_{s \to 1}
				\frac{
					L(M, s)
				}{
					(s - 1)^{r_{M}}
				}
		=
			(-1)^{\rank X(K)} \cdot
			\chi_{W}(S, \mathcal{M}^{\Delta})^{-1} \cdot
			q^{\chi(S, \Lie \mathcal{G}^{0})} \cdot
			(\log q)^{r_{M}}.
	\]
Here $\chi_{W}(S, \mathcal{M}^{\Delta})$ is
the Euler characteristic of the complex $H^*_{W}(S, \mathcal{M}^{\Delta})$
with differential the cup product with a generator $e\in H^1_W(S,\Z)\cong \Z$.
\end{Thm}

This includes the formula for abelian varieties in \cite{GS20}
and implies a formula for tori as a special case. 
Note that the left-hand side depends on the map $X\to G$
just as the right-hand side does;
in fact, we have
	\[
			L(M, s)
		=
			L(\mathcal{X}^{\Delta}, s - 1) \cdot
			L(T, s) \cdot L(A, s),
	\]
where $L(\mathcal{X}^{\Delta}, s)$ is the $L$-function of $\mathcal{X}^{\Delta}$ defined below.
In particular, the theorem is more subtle than just 
combining formulas for abelian varieties, tori, and lattices, and we need 
a result for $\Z$-constructible sheaves $\mathcal Z$ on $S$. Let
	\[L(\mathcal{Z}, s)=
		\prod_{v}
			\det(1 - \varphi_{v} N(v)^{-s} \,|\, \mathcal{Z}_{\closure{v}} \tensor_{\Z} \Q_l)^{-1},
	\]
be the $L$-function of the $l$-adic sheaf $\mathcal{Z} \tensor_{\Z} \Q_{l}$ 
on $S$, where $\mathcal{Z}_{\closure{v}}$ is the stalk of $\mathcal{Z}$ at a 
geometric point lying over $v$.
The Weil-\'etale cohomology groups $H^*_{W}(S, \mathcal{Z})$ are finitely
generated, and we define $r_{\mathcal{Z}}$ as in 
\eqref{rm} and $\chi_{W}(S, \mathcal{Z})$ as above. 

\begin{Thm} \label{thmconstructible}
The function $L(\mathcal{Z}, s)$ has a pole of order $r_{\mathcal{Z}}$
at $s=0$, and
		$$
				\lim_{s \to 0}
					L(\mathcal{Z}, s) \cdot s^{r_{\mathcal{Z}}}
			= (-1)^{\rank \mathcal{Z}(K)} \cdot \chi_{W}(S, \mathcal{Z}) \cdot
				(\log q)^{- r_{\mathcal{Z}}}.
		$$
\end{Thm}
Similar formulas in the number field case were given
by Tran \cite{Tra15}, \cite{Tra16}. The proof uses Artin's induction 
theorem to reduce to the case of $\mathcal{Z}=\Z$. 
To prove Theorem \ref{mainthm} for $M=T$ a torus, we apply 
Theorem \ref{thmconstructible} to $j_*Y$, where $Y=\Hom(T,\Gm)$ is the character module of $T$,
and use duality for Weil-\'etale cohomology of \cite{Gei12} as well as
the functional equation
	\[
			L(T, 1 - s)
		=
			q^{- \chi(S, \Lie(\mathcal{T}^{0})) (2 s - 1)} \cdot
			L(j_*Y, s).
	\]
The proof of Theorem \ref{mainthm} is completed
by combining the cases of constructible sheaves, tori, and abelian 
varieties.

As a by-product, we are able to express the Ono-Oesterl\'e Tamagawa number 
$ \tau(T) $ of a torus in terms of global invariants:
	\[
			\tau(T)
		=
			\frac{
				\# \Cl(\mathcal{T}^{0})_{\tor} \cdot q^{\chi(S, \Lie \mathcal{T}^{0})}
			}{
				\# \mathcal{T}^{0}(S) \cdot \rho(T) \cdot (\log q)^{\rank Y(K)} \cdot \Disc(h_{T})
			}
	\]
and reprove the Tamagawa number formula of 
Ono \cite{Ono63} and Oesterl\'e  \cite{Oes84}
$$\tau(T)=\frac{\# H^{1}(K, Y)}{\# \Sha(T)}.$$
Here $\rho(T)$ is the value in Theorem \ref{mainthm} for $M=T$,  
	\[\Cl(\mathcal{T}^{0})= 
			\frac{T(\Adele_{K})}{T(K) + \mathcal{T}^{0}(\Order_{\Adele_{K}})}
		\cong
			\frac{\bigoplus_{v} \pi_{0}(\mathcal{T}_{v})(k(v))}{T(K)},
	\]
and $h_T$ is the pairing
	\[
			h_{T}
		\colon
			\Cl(\mathcal{T}^{0}) \times Y(K)
		\to
			\Cl(\Gm)
		=
			\Pic(S)
		\stackrel{\deg}{\to}
			\Z.
	\]

The object $\mathcal{M}^{\Delta}$ is functorial in $M$.
It is closely related to, but different from,
the N\'eron model $\mathcal{M}$ of $M$ in the sense of \cite{Suz19}, whose 
cohomology groups $H^{i}_{W}(S, \mathcal{M})$ are not finitely generated in general.
We are planing to discuss the duality of $\mathcal{M}^{\Delta}$ and its 
Weil-\'etale cohomology,
as well as their relations to the functional equation for $L(M, s)$,
in a forthcoming paper.

It would be desirable to unify Theorems \ref{mainthm} and \ref{thmconstructible}
in terms of ``constructible $1$-motives'' and their $L$-functions.
As a first step, 
Pepin Lehalleur \cite{PL19} defined constructible $1$-motives with $\Q$-coefficients,
but one would need to define a refinement with $\Z$-coefficients
in order to formulate a special value formula, and this is especially difficult 
for the $p$-integral structure. 

After the first version of this paper was uploaded to the arXiv,
A.~Morin \cite{Mor22} gave a number field version of 
Theorem \ref{thmconstructible}, improving on Tran's work.

\begin{Notation}
	Throughout the paper $k = \F_{q}$ is a finite field of characteristic 
	$p$ and $S$ a proper, smooth, and geometrically connected curve over $k$
	of genus $g$ with function field $K$.
	For a place $v$ of $K$ (or a closed point of $S$),
	we denote the completed, henselian, and strict henselian local ring of 
	$S$ at $v$ by $\Order_{v}, \Order_{v}^{h}$, and $\Order_{v}^{sh}$,
	and their fraction field by $K_{v}, K_{v}^{h}$, and $K_{v}^{sh}$,
	respectively. 
	Denote the residue field of $\Order_{v}$ by $k(v)$,
	the degree of $v$ by $\deg(v) = [k(v) : k]$, and 
	$N(v) = \# k(v) = q^{\deg(v)}$.
	The adele ring of $K$ is denoted by $\Adele_{K}$
	and its subring of integral adeles by $\Order_{\Adele_{K}}$.
	
	For an abelian group $G$, denote its torsion part by $G_{\tor}$
	and its torsion-free quotient by $G / \tor$.
	If we have a pairing $\varphi \colon G \times H \to \Z$
	between finitely generated abelian groups,
	then the discriminant of $\varphi / \tor \colon G / \tor \times H / \tor \to \Z$
	is denoted by $\Disc(\varphi)$.
	A lattice over a field is a finitely generated free abelian group
	equipped with a continuous action of the absolute Galois group of the 
	field (which necessarily factors through a finite group).
	
	The N\'eron models we consider are N\'eron lft (locally finite type) models
	in the terminology of \cite[Chapter 10]{BLR90}.
	The connected N\'eron model means the part of the N\'eron model with connected fibers,
	usually denoted by $\mathcal{G}^{0}$ if $\mathcal{G}$ denotes the N\'eron model.
	For a group scheme $G$ locally of finite type over a field,
	we denote the \'etale group scheme of connected components of $G$ by $\pi_{0}(G)$
	(\cite[Chapter II, \S 5, No.\ 1, Proposition 1]{DG70}).
\end{Notation}

\begin{Ack}
	The authors are grateful to Kazuya Kato, Stephen Lichtenbaum and Takeshi Saito
	for helpful discussions.
\end{Ack}


\section{Weil etale cohomology of tori and lattices}
\label{sec: Weil etale cohomology of tori and lattices}

We recall the Weil-\'etale cohomology groups \cite{Lic05}, \cite{Gei04}
of varieties over a finite field.
See also \cite[Section 5]{GS20} for another survey.

Let $\closure{k} = \closure{\F}_{q}$ be an algebraic closure of $k = \F_{q}$.
Denote the $q$-th power (arithmetic) Frobenius map by
$\phi \in \Gal(\closure{k} / k)$.
Denote the Weil group of $k$ by
$W = \langle \phi \rangle \subset \Gal(\closure{k} / k)$
and the category of $W$-modules by $\Mod_{W}$.
Let $X$ be a proper smooth variety over $k$
and $\closure{X}$ be 
the base change of $X$ to $\closure{k}$.
Denote the category of sheaves of abelian groups on the small \'etale site $X_{\et}$
by $\Ab(X_{\et})$ and its bounded derived category by $D^b(X_{\et})$.
For a sheaf $\mathcal{F} \in \Ab(X_{\et})$,
we denote its inverse image to $\closure{X}_{\et}$ by $\closure{\mathcal{F}}$.
Hence we have a left exact functor
$$\Ab(X_{\et}) \to \Mod_{W},\quad  \mathcal{F} \mapsto \Gamma(\closure{X}, \closure{\mathcal{F}})$$
and its right derived functor
$$D^b(X_{\et}) \to D^b(\Mod_{W}),\quad 
\mathcal{F}^{\cdot} \mapsto R \Gamma(\closure{X}, 
\closure{\mathcal{F}}^{\cdot}).$$
The group $\Gamma(\closure{X}, \closure{\mathcal{F}})$ has a natural action 
of $\phi$.
Composing it with the group cohomology functor $R \Gamma(W, \;\cdot\;)$,
we obtain a triangulated functor 
	\begin{align*}
			R \Gamma_{W}(X, \;\cdot\;)
		\colon
			D^b(X_{\et})
		&\to
			D^b(\Ab),\\		
			\mathcal{F}^{\cdot}
		&\mapsto
			R \Gamma(W, R \Gamma(\closure{X}, \closure{\mathcal{F}}^{\cdot}).
	\end{align*}
and denote the $n$-th cohomology of $R \Gamma_{W}(X, \;\cdot\;)$
by $H_{W}^{n}(X, \;\cdot\;)$.

Let $e \in H_{W}^{1}(k, \Z) \cong \Hom(W, \Z)$ be the homomorphism
sending $\phi$ to $1$.
For any $\mathcal{F}^{\cdot} \in D^b(X_{\et})$,
the cup product with $e$ defines a homomorphism
$e \colon H_{W}^{n}(X, \mathcal{F}^{\cdot}) \to H_{W}^{n + 1}(X, \mathcal{F}^{\cdot})$.
Since $e \cup e = 0$, we obtain a complex
$(H_{W}^{\ast}(X, \mathcal{F}^{\cdot}), e)$.
It is exact after tensoring with $\Q$ by \cite[Corollary 5.2]{Gei04}, 
hence the cohomology groups of this complex are torsion.
If the groups $H_{W}^{n}(X, \mathcal{F}^{\cdot})$ are finitely generated for all $n$
and zero for almost all $n$,
then the Euler characteristic of the complex 
$(H_{W}^{\ast}(X, \mathcal{F}^{\cdot}), e)$ is thus well-defined.
We denote this Euler characteristic by $\chi_{W}(X, \mathcal{F}^{\cdot})$:
	\[
			\chi_{W}(X, \mathcal{F}^{\cdot})
		=
			\chi(H_{W}^{\ast}(X, \mathcal{F}^{\cdot}), e)
		=
			\prod_{n}
				\bigl(
					\# H^{n}(H_{W}^{\ast}(X, \mathcal{F}^{\cdot}), e)
				\bigr)^{(-1)^{n}}.
	\]
Weil-\'etale cohomology does not depend on the choice of $k$:

\begin{Prop} \label{prop: independence of the constant field of Weil etale cohom}
	Let $k' / k$ be a finite field contained in the field of constants of $X$,
	and let $W$ and $W'$ be the Weil groups of $k$ and  $k'$, respectively.
	Then the natural morphism
		\[
				R \Gamma_{W}(X, \;\cdot\;)
			\to
				R \Gamma_{W'}(X, \;\cdot\;)
		\]
	of functors $D^b(X_{\et}) \to D^b(\Ab)$ is an isomorphism.
\end{Prop}

\begin{proof}
	Fix an algebraic closure $\closure{k'} = \closure{k}$ of $k'$ (or $k$).
	Then $R \Gamma(X \times_{k} \closure{k}, \;\cdot\;)$
	is canonically isomorphic to
	$R \Gamma(X \times_{k'} \closure{k'}, \;\cdot\;) \tensor_{\Z[W']}^{L} \Z[W]$
	as functors $D^b(X_{\et}) \to D^b(\Mod_{W})$.
	This implies the result.
\end{proof}

Note however that the cup product with a generator of $H_{W}^{1}(k, \Z)$
and a generator of $H_{W'}^{1}(k', \Z)$ on these isomorphic functors 
differ by a factor of $[k':k]$.

From now on we assume that the base is a smooth, proper, and 
geometrically connected curve $S$ over $k$.
Let $K$ be the function field of $S$,  
$T / K$ be a torus and $Y=\Hom_{\bar K}(T,\Gm)$ its character lattice.
Let $\mathcal{T}$ and\ $\mathcal{T}^{0}$ be the 
N\'eron and connected N\'eron models over $S$, respectively.
Let $\mathcal{Y} = j_{\ast} Y$ and 
$\Tilde{\mathcal{Y}} = \tau_{\le 1} R j_{\ast} Y$,
where $j\colon \Spec K \into S$ is the inclusion
and $\tau_{\le 1}$ is the truncation functor in degrees $\le 1$
(in the cohomological grading).

By \cite[Definition 4.8]{Suz19}, the natural pairing 
$T \times Y \to \Gm$ over $K$ canonically extends to a morphism
\begin{equation}\label{ppaaiirr}
			\mathcal{T}^{0} \tensor^{L} \Tilde{\mathcal{Y}}
		\to
			\Gm
\end{equation}
in $D^b(S_{\et})$.
Denote the sheaf-Hom functor for $S_{\et}$ by
$\sheafhom_{S_{\et}}$.

\begin{Prop}\label{thesame}
	The induced morphism
		\[
				\mathcal{T}^{0}
			\to
				R \sheafhom_{S_{\et}}(\Tilde{\mathcal{Y}}, \Gm)
		\]
	in $D^b(S_{\et})$ is an isomorphism.
\end{Prop}

\begin{proof}
We can check this at stalks. The morphism pulled back to $K_{\et}$ is
nothing but the duality between $T$ and $Y$.
Hence it is enough to show that for any place $v \in S$, the induced morphism
		\[
				\mathcal{T}^{0}(\Order_{v}^{sh})
			\to
				R \Hom_{\Order_{v, \et}^{sh}}(\Tilde{\mathcal{Y}}, \Gm)
		\]
in $D^b(\Ab)$ is an isomorphism.
Denote by $j \colon \Spec K_{v}^{sh} \into \Spec \Order_{v}^{sh}$
	and $i \colon \Spec \closure{k(v)} \into \Spec \Order_{v}^{sh}$
	the inclusions.
	Set $\mathcal{Y}^{0} = j_{!} Y$
	and $\Tilde{\mathcal{Y}}_{\closure{v}} = i^{\ast} \Tilde{\mathcal{Y}} = \tau_{\le 1} R \Gamma(K_{v}^{sh}, Y)$.
	By \cite[Proposition 4.14]{Suz19},
	we have a canonical morphism of distinguished triangles
		\[
			\begin{CD}
					\mathcal{T}^{0}(\Order_{v}^{sh})
				@>>>
					T(K_{v}^{sh})
				@>>>
					\pi_{0}(\mathcal{T}_{v})(\closure{k(v)})
				\\ @VVV @VVV @VVV \\
					R \Hom_{\Order_{v, \et}^{sh}}(\Tilde{\mathcal{Y}}, \Gm)
				@>>>
					R \Hom_{\Order_{v, \et}^{sh}}(\mathcal{Y}^{0}, \Gm)
				@>>>
					R \Hom_{\Order_{v, \et}^{sh}}(i_{\ast} \Tilde{\mathcal{Y}}_{\closure{v}}, \Gm)[1].
			\end{CD}
		\]
(There is actually a shifted term $\mathcal{T}^{0}(\Order_{v}^{sh})[1]$
next to $\pi_{0}(\mathcal{T}_{v})(\closure{k(v)})$,
a similar term for the lower row
and another commutative square next to the right square.)
By the adjunction 
and the duality between $T$ and $Y$, the middle map can be identified
with the isomorphism
		\[
		T(K_{v}^{sh})
		\cong
		R \Gamma(K_{v}^{sh}, T)
		\cong
				R \Hom_{K_{v, \et}^{sh}}(Y, \Gm)
		\]
since $H^{n}(K_{v}^{sh}, T) = 0$ for $n \ge 1$ by
\cite[Chapter X, Section 7, ``Application'']{Ser79}.
For the right vertical morphism, we use the exact sequence
		\[
			0 \to \Gm \to \mathcal{G}_{m} \to i_{\ast} \Z \to 0
		\]
	in $\Ab(\Order_{v}^{sh})$, 
where $\mathcal{G}_{m}$ is the N\'eron model of $\Gm$.
Since $H^{n}(K_{v}^{sh}, \Gm) = 0$  for $n \ge 1$, we have 
$R j_{\ast} \Gm \cong \mathcal{G}_{m}$, hence
		\[
				R \Hom_{\Order_{v, \et}^{sh}}(i_{\ast} \Tilde{\mathcal{Y}}_{\closure{v}}, \mathcal{G}_{m})
			\cong
				R \Hom_{K_{v, \et}^{sh}}(j^{\ast} i_{\ast} \Tilde{\mathcal{Y}}_{\closure{v}}, \Gm)
			=
				0.
		\]
	Therefore
		\begin{align*}
					R \Hom_{\Order_{v, \et}^{sh}}(i_{\ast} \Tilde{\mathcal{Y}}_{\closure{v}}, \Gm)[1]
			&	\isomfrom
					R \Hom_{\Order_{v, \et}^{sh}}(i_{\ast} \Tilde{\mathcal{Y}}_{\closure{v}}, i_{\ast} \Z)
			\\
			&	\cong
					R \Hom_{\Ab}(\Tilde{\mathcal{Y}}_{\closure{v}}, \Z)
			\\
			&	=
					R \Hom_{\Ab}(\tau_{\le 1} R \Gamma(K_{v}^{sh}, Y), \Z).
		\end{align*}
	Therefore the right vertical morphism in the above diagram is
		\[
				\pi_{0}(\mathcal{T}_{v})(\closure{k(v)})
			\to
				R \Hom_{\Ab}(\tau_{\le 1} R \Gamma(K_{v}^{sh}, Y), \Z).
		\]
This is an isomorphism by \cite[Theorem B (5)]{Suz19}.
Therefore the left vertical morphism in the above diagram is also 
an isomorphism.
\end{proof}

\begin{Thm}\label{thm: global duality for tori}
\label{prop: cohom of torus is finitely generated and bounded}
The groups $H_{W}^{n}(S, \mathcal{T}^{0})$ as well as the group 
$H_{W}^{n}(S, \Tilde{\mathcal{Y}})$ are finitely 
generated for all $n$, and zero for $n \ne 0, 1, 2, 3$.
Moreover, the pairing 	
\[
			R \Gamma_{W}(S, \mathcal{T}^{0}) \tensor^{L} 
			R \Gamma_{W}(S, \Tilde{\mathcal{Y}})
		\to
			R \Gamma_{W}(S, \Gm)
		\to
			\Z[-2]
	\]
induced by \ref{ppaaiirr} is perfect.
\end{Thm}

\begin{proof}
This follows from 
\cite[Proposition 2.6, Theorem 4.2 and Corollary 4.3]{Gei12} because the 
cohomology groups and the pairing agrees with the one 
in \cite{Gei12} by Proposition \ref{thesame}.
\end{proof}

\begin{Cor} \label{cor: global duality for tori, each term}
We have perfect pairings
		\[
					H_{W}^{n}(S, \mathcal{T}^{0}) / \tor
				\times
					H_{W}^{2 - n}(S, \Tilde{\mathcal{Y}}) / \tor
			\to
				\Z
		\]
of finitely generated free abelian groups as well as perfect pairings
		\[
					H_{W}^{n}(U, \mathcal{T}^{0})_{\tor}
				\times
					H_{W}^{3 - n}(U, \Tilde{\mathcal{Y}})_{\tor}
			\to
				\Q / \Z
		\]
of finite abelian groups.
\end{Cor}

\begin{Lem}
We have 
$$\dim H^{0}(S, \Tilde{\mathcal{Y}}) \tensor \Q =\rank  Y(K),$$
and all other \'etale cohomology groups of $\Tilde{\mathcal{Y}}$ are torsion.
\end{Lem}

\begin{proof}
The first statement follows from
$H^{0}(S, \Tilde{\mathcal{Y}}) =H^{0}(K,Y)$, and second from 
the long exact sequence
	\[
		\ldots\to	H^i(S, \Tilde{\mathcal{Y}})\to H^i(K,Y)\to
			H^i(S, \tau_{>1}Rj_* Y)\to \ldots
	\]
because the sheaves $R^nj_*Y$ as well as Galois cohomology are 
torsion for $n>0$. 
\end{proof}

\begin{Prop}\label{prop: finiteness and freeness for tori and lattices}
\begin{align*}
\rank H^i_W(S, \Tilde{\mathcal{Y}}) &=
\begin{cases}
\rank Y(K)& i=0,1\\
0&\text{otherwise}
\end{cases}\\
\rank H^i_W(S,\mathcal{T}^{0}) &=
\begin{cases}
\rank Y(K)& i=1,2\\
0&\text{otherwise}.
\end{cases}
\end{align*}
\end{Prop}

\begin{proof} 
We know the groups in question are finitely generated. The statement for
$H^i_W(S, \Tilde{\mathcal{Y}})$ follows from the Lemma and
the isomorphism between $H^i_W(S, \Tilde{\mathcal{Y}}) \tensor \Q$
and $(H^i(S, \Tilde{\mathcal{Y}}) \oplus H^{i-1}(S, \Tilde{\mathcal{Y}})) \tensor \Q$
(\cite[Corollary 5.2]{Gei04}).
The statement for  
$H^i_W(S,\mathcal{T}^{0})$ follows by duality.
\end{proof}

We summarize the above finiteness results
(where fg stands for finitely generated):
\begin{center}
	\begin{tabular}{c|cccc}
		$i$ & $0$ & $1$ & $2$ & $3$ \\ \hline
		$H_{W}^{i}(S, \Tilde{\mathcal{Y}})$ \rule{0pt}{13pt} & fg free & fg & finite & finite \\
		$H_{W}^{i}(S, \mathcal{T}^{0})$ & finite & fg & fg & $0$ \\
	\end{tabular}
\end{center}
Here the freeness of $Y(K)$ reflects to the vanishing of
(the torsion of) $H_{W}^{3}(S, \mathcal{T}^{0})$.

\begin{Cor} \label{weighted}
The secondary Euler characteristic of the torus and the lattice are given by
\begin{align*}
r_T&:=-\sum_i (-1)^i \cdot i \cdot \rank H^i_W(S,\mathcal{T}^{0})=-\rank Y(K)\\
r_Y&:=- \sum_i (-1)^i \cdot i \cdot \rank H^i_W(S, \Tilde{\mathcal{Y}})=
\rank Y(K)
\end{align*}
\end{Cor}

\begin{Prop} \label{prop: duality and Euler characteristic}
	\[
		\chi_{W}(S, \mathcal{T}^{0}) = \chi_{W}(S, \Tilde{\mathcal{Y}})^{-1}.
	\]
\end{Prop}

\begin{proof}
	The cup product with the element $e \in H_{W}^{1}(k, \Z)$ induces a natural transformation
	$e \colon R \Gamma_{W}(S, \;\cdot\;) \to R \Gamma_{W}(S, \;\cdot\;)[1]$.
	The associativity of cup product shows that the diagram
		\[
			\begin{CD}
						R \Gamma_{W}(S, \mathcal{T}^{0})
					\tensor^{L}
						R \Gamma_{W}(S, \Tilde{\mathcal{Y}})
				@> \id \tensor e >>
						R \Gamma_{W}(S, \mathcal{T}^{0})
					\tensor^{L}
						R \Gamma_{W}(S, \Tilde{\mathcal{Y}})[1]
				\\
				@VV e \tensor \id V @VV \cup V
				\\
						R \Gamma_{W}(S, \mathcal{T}^{0})[1]
					\tensor^{L}
						R \Gamma_{W}(S, \Tilde{\mathcal{Y}})
				@> \cup >>
					R \Gamma_{W}(S, \mathcal{T}^{0} \tensor^{L} \Tilde{\mathcal{Y}})[1]
			\end{CD}
		\]
	is commutative.
	Hence the diagram
		\[
			\begin{CD}
					R \Gamma_{W}(S, \mathcal{T}^{0})
				@> \sim >>
					R \Hom(R \Gamma_{W}(S, \Tilde{\mathcal{Y}})[2], \Z)
				\\
				@VV e V @VV e V
				\\
					R \Gamma_{W}(S, \mathcal{T}^{0})[1]
				@> \sim >>
					R \Hom(R \Gamma_{W}(S, \Tilde{\mathcal{Y}})[1], \Z)
			\end{CD}
		\]
	is commutative, from which the result follows.
\end{proof}


\section{$L$-values for $\Z$-constructible sheaves}
\label{sec: L values for Z constructible sheaves}

Let $\mathcal{Z}$ be a $\Z$-constructible \'etale sheaf on $S$
\cite[Chapter II, Section 0]{Mil06}.
The groups $H_{W}^{\ast}(S, \mathcal{Z})$ are finitely generated by
\cite[Proposition 2.6]{Gei12}. 
In this section, we prove a formula for the leading coefficient of
the $L$-function of $\mathcal{Z}$ at $s = 0$
in terms of $\chi_{W}(S, \mathcal{Z})$.
Similar formulas in the number field case were given
by Tran \cite{Tra15}, \cite{Tra16}.

We define $L(\mathcal{Z}, s)$ by the Euler product
	\[
		\prod_{v}
			\det(1 - \varphi_{v} N(v)^{-s} \,|\, \mathcal{Z}_{\closure{v}} \tensor_{\Z} \Q)^{-1},
	\]
where $v$ runs through the places of $S$,
$\mathcal{Z}_{\closure{v}}$ is the stalk of $\mathcal{Z}$ at a geometric point lying over $v$
and $\varphi_{v}$ is the geometric Frobenius at $v$.
This function agrees with the $L$-function
of the $l$-adic sheaf $\mathcal{Z} \tensor_{\Z} \Q_{l}$ on $S$
\cite[Definition 5.3.7]{Kah18},
where $l \ne p$ is any prime.
In particular, it is equal to the value at $t = q^{-s}$ of the rational function
	\[
		\prod_{i = 0}^{2}
			\det(
				1 - \varphi t
			\,|\,
				H^{i}(\closure{S}, \closure{\mathcal{Z}} \tensor_{\Z} \Q_{l})
			)^{(-1)^{i + 1}},
	\]
where $\varphi$ is the geometric Frobenius of $k$
and the cohomology is taken as the continuous cohomology
\cite[Corollary 5.3.11]{Kah18}.
Let $j \colon \Spec K \into S$ be the inclusion,
$Y$ the torsion-free quotient of the generic fiber $j^{\ast} \mathcal{Z}$ of $\mathcal{Z}$
and $Y'$ the dual lattice of $Y$.
We set
	\[
			L(Y, s)
		:=
			L(j_{\ast} Y, s).
	\]
Then $L(Y, s)$ agrees with the Artin $L$-function 
	\[
		\prod_{v}
	\det(1 - \phi_{v} N(v)^{-s} \,|\, Y'(K_{v}^{sh}) \tensor_{\Z} \Q)^{-1},
	\]
of $Y'$, where $\phi_{v}$ is the arithmetic Frobenius at $v$.
The function $L(Y, s)$ is also the Hasse-Weil $L$-function 	
	\[
		\prod_{v}
			\det \bigl(
					1 - \varphi_{v} N(v)^{-s}
				\,|\,
					(Y(K^{\sep}) \tensor_{\Z} \Q_{l})^{I_{v}}
			\bigr)^{-1},
	\]
of the $l$-adic representation $Y(K^{\sep}) \tensor_{\Z} \Q_{l}$ over $K$, 	
where $I_{v}$ is the inertial group at $v$.
The functions $L(\mathcal{Z}, s)$ and $L(Y, s)$ are equal
up to finitely many Euler factors
since the kernel and the cokernel of the natural morphism
$\mathcal{Z} \to j_{\ast} Y$
are concentrated at finitely many closed points.
Let 
$$r_{\mathcal{Z}}=- \sum i\cdot (-1)^i\cdot \rank H^i_W(S,\mathcal{Z}).$$

\begin{Thm} \label{thm: L value formula for Z constructible sheaf}
    We have
		\begin{equation} \label{eq: L value formula for Z constructible sheaf}
				\lim_{s \to 0}
					L(\mathcal{Z}, s) \cdot s^{r_{\mathcal{Z}}}
			=
				\pm
				\chi_{W}(S, \mathcal{Z}) \cdot
				(\log q)^{- r_{\mathcal{Z}}}.
		\end{equation}
\end{Thm}

We will see in Proposition 
\ref{prop: sign of L value for Z constructible sheaf at zero}
that the sign $\pm$ is $(-1)^{\rank \mathcal{Z}(K)}$.

\begin{proof} 
We will proceed in several steps.

\medskip
\noindent{\bf Step 1:}
If $\mathcal{Z} = \Z$, then $L(\mathcal{Z}, s)$ is the zeta function of $S$,
and \eqref{eq: L value formula for Z constructible sheaf} reduces to
\cite[Theorem 9.1, Proposition 9.2]{Gei04}.

\medskip
\noindent{\bf Step 2:} 
If $\mathcal{Z}$ is constructible, then both sides of
\eqref{eq: L value formula for Z constructible sheaf} are $1$.

\smallskip
This is clear for the left-hand side. For the right-hand side, 
the constructibility of $\mathcal{Z}$ implies that
the groups $H^{\ast}(\Bar{S}, \mathcal{Z})$ are finite
by [Mil80, Chapter VI, Theorem 2.1], hence
$\# H^{\ast}(\Bar{S}, \mathcal{Z})^G = \# H^{\ast}(\Bar{S}, \mathcal{Z})_G$,
which implies that $\chi_{W}(S, \mathcal{Z}) = 1$ in view of the short exact
sequences
$$0\to H^{i-1}(\Bar{S}, \mathcal{Z})_G\to H^i_W(S, \mathcal{Z})\to 
 H^{i}(\Bar{S}, \mathcal{Z})^G\to 0.$$

\medskip
\noindent{\bf Step 3:} 
Equation \eqref{eq: L value formula for Z constructible sheaf} holds
if $\mathcal{Z}$ is supported on closed points of $S$.

\smallskip

	We may assume that $\mathcal{Z}$ is supported at a single place $v$.
	Write $\mathcal{Z} = i_{v \ast} \mathcal{Z}_{v}$,
	where $i_{v} \colon v \into S$ is the inclusion.
	The number $r := r_{\mathcal{Z}}$ is the rank of $\mathcal{Z}_{v}(k(v))$.
	The left-hand side is
		\begin{align*}
			&
						\lim_{s \to 0}
							\frac{
								(1 - N(v)^{-s})^{r}
							}{
								\det \bigl(
									1 - \varphi_{v} N(v)^{-s} \;|\; \mathcal{Z}_{\closure{v}} \tensor \Q
								\bigr)
							}
					\cdot
						(\deg(v) \cdot \log q)^{- r}
			\\
			&	=
						\chi_{W_{v}}(k(v), \mathcal{Z}_{v})
					\cdot
						(\deg(v) \cdot \log q)^{- r},
		\end{align*}
	where $W_{v}$ is the Weil group of $k(v)$.
	Let $e_{v} \in H_{W_{v}}^{1}(k(v), \Z)$ be the generator
	corresponding to the arithmetic Frobenius of $k(v)$. Since 
	$H_{W}^{1}(k, \Z) \to H_{W_{v}}^{1}(k(v), \Z)$ is multiplication by 
	$\deg(v)$, we obtain
		\[
				 \chi_{W}(S, \mathcal{Z})
			=
				\chi_{W}(k(v), \mathcal{Z}_{v})
			=
					\chi_{W_{v}}(k(v), \mathcal{Z}_{v})
				\cdot
					\deg(v)^{- r}.
		\]
	Hence both sides of \eqref{eq: L value formula for Z constructible sheaf}
	are $\chi_{W}(S, \mathcal{Z}) \cdot (\log q)^{- r}$.

\medskip
\noindent{\bf Step 4:} 
	Let $K' / K$ be a finite separable extension.
	Denote the normalization of $S$ in $K'$ by $S'$.
	Assume that $\mathcal{Z}$ is the pushforward of a $\Z$-constructible 
	sheaf $\mathcal{Z}'$ on $S'$.
	Then \eqref{eq: L value formula for Z constructible sheaf} 
	for $\mathcal{Z}'$ over $S'$ implies 
	\eqref{eq: L value formula for Z constructible sheaf}
	for $\mathcal{Z}$ over $S$.

\smallskip
	Denote the constant field of $K'$ by $k'$, its order by $q'$, and
	the morphism $S' \to S$ by $\pi$.
	Since $\pi$ is finite, we have $R^{n} \pi_{\ast}= 0$ for $n \ge 1$.
	Hence by \cite[Chapter VI, Lemma 13.8 (c)]{Mil80}, we have
		\[
				\det(
					1 - \varphi t
				\,|\,
					H^{i}(\closure{S}, \mathcal{Z} \tensor_{\Z} \Q_{l})
				)
			=
				\det(
					1 - \varphi t
				\,|\,
					H^{i}(\closure{S'}, \mathcal{Z}' \tensor_{\Z} \Q_{l})
				)
		\]
	for any $i$,
	where $\closure{S'} = S' \times_{k} \closure{k}$ as before.
	Taking the alternating product over $i$, evaluating it at $t = q^{-s}$
	and noting that the $L$-function is independent of the choice of a constant field,
	we have $L(\mathcal{Z}, s) = L(\mathcal{Z}', s)$.
	
	Also $H_{W}^{\ast}(S, \mathcal{Z}) \cong H_{W'}^{\ast}(S', \mathcal{Z}')$
	by Proposition \ref{prop: independence of the constant field of Weil etale cohom},
	where $W'$ is the Weil group of $k'$.
	In particular, $r_{\mathcal{Z}} = r_{\mathcal{Z}'}$.
	Let $e' \in H_{W'}^{1}(k', \Z)$ be the generator
	corresponding to the $q'$-th power arithmetic Frobenius.
	Then $e = [k' : k] e'$ via the homomorphism $H_{W}^{1}(k, \Z) \to H_{W'}^{1}(k', \Z)$.
	This implies
		\[
				\chi_{W}(S, \mathcal{Z})
			=
					\chi_{W'}(S', \mathcal{Z}')
				\cdot
					[k' : k]^{- r_{\mathcal{Z}'}}.
		\]
	As $q' = q^{[k' : k]}$, we get the result.

\medskip
\noindent{\bf Step 5:} 
	Let $n \ge 1$ be an integer.
	If \eqref{eq: L value formula for Z constructible sheaf} holds for
    for $\mathcal{Z}^{n}$, 
	then it holds for $\mathcal{Z}$.

\smallskip
This follows by taking the $n$-th roots of the real numbers on 
both sides of \eqref{eq: L value formula for Z constructible sheaf}.

\medskip
\noindent{\bf Step 6:}  
 	We now finish the proof. We first observe that both of the sides 
	of \eqref{eq: L value formula for Z constructible sheaf}
	are multiplicative in $\mathcal{Z}$ with respect to short exact sequences.
	In particular, we may assume that $\mathcal{Z}$ is torsion-free
	by Step 2. Denote the generic fiber of $\mathcal{Z}$ by $\mathcal{Z}_{K}$.
	By the Artin induction theorem \cite[Corollary 4.4, Proposition 4.1]{Swa60},
	there exist an integer $n \ge 1$,
	finite separable extensions $K'_{1}, \dots, K'_{m}$, $K''_{1}, \dots, K''_{l}$,
	a finite Galois module $N_{K}$ over $K$
	and an exact sequence
		\[
				0
			\to
					\mathcal{Z}_{K}^{n}
				\oplus
					\bigoplus_{i} \pi_{K'_{i} / K, \ast} \Z
			\to
				\bigoplus_{j} \pi_{K''_{j} / K, \ast} \Z
			\to
				N_{K}
			\to
				0
		\]
	of Galois modules over $K$,
	where $\pi_{K'_{i} / K}$ is the morphism $\Spec K'_{i} \to \Spec K$
	and $\pi_{K''_{j} / K}$ are similarly defined.
	By spreading out, this sequence can be obtained as
	the generic fiber of an exact sequence
		\begin{equation} \label{eq: Artin induction over U}
				0
			\to
					\mathcal{Z}_{U}^{n}
				\oplus
					\bigoplus_{i} \pi_{U'_{i} / U, \ast} \Z
			\to
				\bigoplus_{j} \pi_{U''_{j} / U, \ast} \Z
			\to
				N_{U}
			\to
				0
		\end{equation}
	of \'etale sheaves over some dense open subscheme $U \subseteq S$,
	where  $\mathcal{Z}_{U}$ is the restriction of $\mathcal{Z}$ to $U$;
	$U'_{i}$ and  $U''_{j}$ denote the normalization of $U$ in
	$K'_{i}$ and $K''_{j}$, respectively;
	$\pi_{U'_{i} / U}$ and $\pi_{U''_{j} / U}$ are the morphism 
	$U'_{i} \to U$ and $U''_{j} \to U$, respectively;
	and $N_{U}$ is a finite \'etale group scheme over $U$.
	Denote the inclusion map $U \into S$ by $\iota$.
	By the exact sequence
		\[
				0
			\to
				\iota_{!} \mathcal{Z}_{U}
			\to
				\mathcal{Z}
			\to
				\bigoplus_{v \notin U}
					i_{v, \ast} \mathcal{Z}_{v}
			\to
				0
		\]
	and Step 3,  
	\eqref{eq: L value formula for Z constructible sheaf}
	for $\iota_{!} \mathcal{Z}_{U}$ and for $\mathcal{Z}$
	are equivalent.
	By the exact sequence \eqref{eq: Artin induction over U},
	Steps 2 and 5 and the exactness of $\iota_{!}$, 
	it is enough to show the proposition for
	$\iota_{!} \pi_{U'_{i} / U, \ast} \Z$ and 
	$\iota_{!} \pi_{U''_{j} / U, \ast} \Z$.
	If $\iota_{i}' \colon U'_{i} \into S'_{i}$ is the inclusion
	into the smooth compactification, then
	$\iota_{!}(\pi_{U'_{i} / U, \ast} \Z) 
	\cong \pi_{S'_{i} / S, \ast}(\iota_{i, !}' \Z)$.
	Finally, \eqref{eq: L value formula for Z constructible sheaf}
	for $\pi_{S'_{i} / S, \ast}(\iota_{i, !}' \Z)$
	follows from Steps  1, 2, and 4. 
\end{proof}


\section{Functional equations and $L$-values for tori}
\label{sec: Functional equations and L values for tori}

We will determine the sign and express the exponential term
appearing in the functional equation relating the $L$-functions of $T$ 
and of $Y$ in terms of $\chi(S, \Lie \mathcal{T}^{0})$.
This will allow us to give a formula for the leading coefficient of the 
$L$-function of $T$ at $s = 1$ in terms of $\chi_{W}(S, \mathcal{T}^{0})$.

Let $r=\rank Y(K)$, $d=\dim T$, 
$Y'$ be the dual lattice of $Y$,
and define
	\[
			L(T, s)
		:=
			L(Y', s).
	\]
More explicitly, 
by the discussion in Section \ref{sec: L values for Z constructible sheaves},
we know that $L(T, s)$ is the Artin $L$-function of $Y$
and also the Hasse-Weil $L$-function of the $l$-adic representation
$Y' \tensor_{\Z} \Q_{l} \cong V_{l}(T)(-1)$ over $K$
(where $l \ne p$ is any prime), 
	\[
			L(T, s)
		=
			\prod_{v}
			\det \bigl(
				1 - \varphi_{v} N(v)^{- s} \,|\, V_{l}(T)(-1)^{I_{v}}
			\bigr)^{-1},
	\]
where $\varphi_{v}$ is the geometric Frobenius at $v$
and $I_{v}$ is the inertia group at $v$.

Let $\mathfrak{f}(Y)$ be the Artin conductor of the Galois representation $Y \tensor_{\Z} \Q$ over $K$ as defined in
\cite[Chapter VI, Section 3]{Ser79}.
It is an effective divisor on $S$.
Denote its multiplicity at a place $v$ by $f(Y|_{D_{v}})$
(where $D_{v}$ is the decomposition group at $v$),
so that
	\[
		\mathfrak{f}(Y) = \sum_{v} f(Y|_{D_{v}}) \cdot v.
	\]
The degree $f(Y)$ of $\mathfrak{f}(Y)$ is given by 
	\[
		f(Y) = \sum_{v} \deg(v) f(Y|_{D_{v}}).
	\]

The functional equation for 
Artin $L$-functions \cite[Theorem 4.4.8]{Kah18} in this case says that
	\begin{equation} \label{eq: functional equation for torus}
			L(T, 1 - s)
		=
			\pm q^{((2 g - 2) d + f(Y)) (s - 1 / 2)} \cdot
			L(Y, s).
	\end{equation}
We have
	\[
			L(Y, s)
		=
			\prod_{i = 0}^{2}
				\det(
					1 - \varphi q^{-s}
				\;|\;
					H^{i}(\closure{S}, \mathcal{Y} \tensor \Q_{l})
				)^{(-1)^{i + 1}}
	\]
by \cite[Corollary 5.3.11]{Kah18},
where $\varphi$ is the geometric Frobenius of $k$.
Each term
	$
		\det(
			1 - \varphi t
		\;|\;
			H^{i}(\closure{S}, \mathcal{Y} \tensor \Q_{l})
		)
	$
is a polynomial with integer coefficients in $t$ with constant term $1$
whose reciprocal roots are Weil $q$-numbers of weight $i$
\cite[Theorem 5.5.9]{Kah18}.

\begin{Prop} \label{prop: numerator of L of lattice is zeta of AV}
	We have
		\[
				\det(
					1 - \varphi t
				\;|\;
					H^{1}(\closure{S}, \mathcal{Y} \tensor \Q_{l})
				)
			=
				\det(
					1 - \varphi t
				\;|\;
					H^{1}(\closure{C}, \Q_{l})
				)
		\]
	for some abelian variety $C$ over $\F_{q}$.
\end{Prop}

\begin{proof}
	Let $K'$ be a finite Galois extension of $K$ with Galois group $G$
	that trivializes $Y$.
	Denote the normalization of $S$ in $K'$ by $S'$.
	Denote the constant field of $K'$ by $k'$.
	Let $U \subseteq S$ be a dense open subscheme
	over which $\mathcal{Y}$ is a lattice.
	Let $U' \subseteq S'$ be the inverse image of $U$ in $S'$.
	Denote the inclusion map $\Spec K' \into S'$ by $j'$.
	Set $\mathcal{Y}' = j'_{\ast} (Y \times_{K} K')$,
	$\closure{S}' = S' \times_{k} \closure{k}$
	and $\closure{U}' = U' \times_{k} \closure{k}$.
	The long exact sequence for cohomology with compact support for
	$\closure{U}' \into \closure{S}'$ yields an exact sequence
		\[
				\bigoplus_{v \in S \setminus U}
					Y(K_{v}^{sh} \tensor_{K} K') \tensor \Q_{l}
			\to
				H_{c}^{1}(\closure{U}', \mathcal{Y}' \tensor \Q_{l})
			\to
				H^{1}(\closure{S}', \mathcal{Y}' \tensor \Q_{l})
			\to
				0.
		\]
	This sequence remains exact after taking $G$-invariants
	since $G$ is finite and the groups are $\Q_{l}$-vector spaces.
	Comparing the resulting exact sequence with
	the similar exact sequence for $\closure{U} \into \closure{S}$, 
	we know that
		\[
				H^{1}(\closure{S}, \mathcal{Y} \tensor \Q_{l})
			\isomto
				H^{1}(\closure{S'}, \mathcal{Y}' \tensor \Q_{l})^{G}.
		\]
	The Jacobian variety $J_{S' / k}$ of $S' / k$ \cite[9.2/3]{BLR90}
	is isomorphic to the Weil restriction of the Jacobian $J_{S' / k'}$
	from $k'$ to $k$. In particular, it is an abelian variety over $k$
	with a natural action of $G$ by group scheme morphisms over $k$.
	Consider the abelian variety $J_{S' / k} \tensor_{\Z} Y(K')$ ($\cong J_{S' / k}^{d}$) over $k$.
	The tensor product of the $G$-actions on $J_{S' / k}$ and on $Y(K')$
	defines a $G$-action on $J_{S' / k} \tensor_{\Z} Y(K')$
	by group scheme morphisms over $k$.
	Let $C$ be the maximal reduced and connected subgroup scheme of
	the $G$-invariant part of $J_{S' / k} \tensor_{\Z} Y(K')$.
	It is an abelian variety over $k$.
	Thus to prove
	the proposition it suffices to observe that the $l$-adic Tate module 
	of $C$ is isomorphic to
	$H^{1}(\closure{S'}, \mathcal{Y}' \tensor\Q_{l})^{G}(1)$ as a $\Gal(\closure{k} / k)$-module and
		\[
				H^{1}(\closure{S}, \mathcal{Y} \tensor \Q_{l})
				\cong
				H^{1}(\closure{S'}, \mathcal{Y}' \tensor \Q_{l})^{G}
			\cong
				H^{1}(\closure{C}, \Q_{l}).
		\]
\end{proof}

\begin{Prop} \label{prop: sing of functional equation}
	The sign of the functional equation \eqref{eq: functional equation for torus} is positive.
\end{Prop}

\begin{proof}
	Recall again that $L(T, s) = L(Y', s)$
	and $r = \rank Y(K) = \rank Y(K')$.
	The function $L(Y, s)$ is real-valued for real $s$, positive for large real $s$ (by the Euler product)
	and has a pole of order $r$ at $s = 0$ and $s = 1$
	by Theorem \ref{thm: L value formula for Z constructible sheaf}
	and \eqref{eq: functional equation for torus}.
	The only other possible zero or pole are at $s = 1 / 2$. 
	Hence it is enough to show that $L(Y, s)$ has a zero of even order at 
	$s = 1 / 2$.
	This order is equal to the order of zero of the function
	$
		\det(
			1 - \varphi q^{-s}
		\;|\;
			H^{1}(\closure{S}, \mathcal{Y} \tensor \Q_{l})
		)
	$
	at $s = 1 / 2$.
	But this function is a polynomial with $\Z$-coefficients in $q^{-s}$
	of even degree by Proposition \ref{prop: numerator of L of lattice is zeta of AV}.
\end{proof}

\begin{Prop} \label{prop: sign of L value for Z constructible sheaf at zero}
	The sign in the formula \eqref{eq: L value formula for Z constructible sheaf}
	is $(-1)^{\rank \mathcal{Z}(K)}$.
\end{Prop}

\begin{proof}
	Denote the generic fiber of $\mathcal{Z}$ by $\mathcal{Z}_{K}$.
	The two functions $L(\mathcal{Z}, s)$ and
	$L(\mathcal{Z}_{K}, s)$ ($= L(j_{\ast} \mathcal{Z}_{K}, s)$)
	differ only by finitely many Euler factors of weight zero
	(namely polynomials in $q^{-s}$ with roots of unity roots).
	Hence they have the same zeros and poles for positive $s$.
	By the proof of Proposition \ref{prop: sing of functional equation},
	we know that the function $L(\mathcal{Z}_{K}, s)$ has a pole
	of order $\rank \mathcal{Z}(K)$ at $s = 1$,
	a zero of even order at $s = 1 / 2$,
	and does not have a zero or pole for other positive values of $s$.
	This implies the result.
\end{proof}

\begin{Prop} \label{prop: conductor formula to torus}
	\[
			\frac{f(Y)}{2}
		=
			- \deg(\Lie \mathcal{T}^{0}).
	\]
\end{Prop}

\begin{proof}
	Recall that $f(Y) = \sum_{v \in S} \deg(v) f(Y|_{D_{v}})$.
	Let $K'$ be a finite Galois extension of $K$ that trivializes $Y$.
	Let $k'$ be the field of constants of $K'$.
	Let $S'$ be the normalization of $S$ in $K'$, 
	and $\mathcal{T}'^{0}$ the connected N\'eron model over $S'$
	of $T \times_{K} K'$.
	For each $v \in S$, fix a place $v'$ of $S'$ above $v$.
	By \cite[Theorem (12.1)]{CY01},
	we have
		\[
				\frac{f(Y|_{D_{v}})}{2}
			=
				\frac{1}{e_{v' / v}}
				\length_{\Order_{v'}}
					\frac{
						\Lie(\mathcal{T}'^{0}) \tensor_{\Order_{S'}} \Order_{v'}
					}{
						\Lie(\mathcal{T}^{0}) \otimes_{\Order_{S}} \Order_{v'}
					},
		\]
	where $e_{v' / v}$ is the ramification index of $\Order_{v'} / \Order_{v}$
	and $\length_{\Order_{v'}}$ denotes the length of $\Order_{v'}$-modules.
	The length of the cokernel of a full rank embedding of finite free modules
	is invariant under taking the top exterior power and inverts when taking 
	duals.
	Hence the right-hand side is equal to
		\[
			\frac{1}{e_{v' / v}}
			\length_{\Order_{v'}}
				\frac{
					\det(\Lie(\mathcal{T}^{0}))^{\ast} \otimes_{\Order_{S}} \Order_{v'}
				}{
					\det(\Lie(\mathcal{T}'^{0}))^{\ast} \tensor_{\Order_{S'}} \Order_{v'}
				},
		\]
	where $\det$ denotes the top exterior power and $\ast$ denotes the dual line bundle.

	Let $\omega \in \det(\Lie(T))^{\ast}(K)$ be
	a non-zero invariant top degree differential form on $T / K$.
	Then $f(Y) / 2$ can be written as the number
		\begin{equation} \label{eq: T part of conductor formula}
				\sum_{v \in S}
					\frac{\deg(v)}{e_{v' / v}}
					\length_{\Order_{v'}}
						\frac{
							\det(\Lie(\mathcal{T}^{0}))^{\ast} \otimes_{\Order_{S}} \Order_{v'}
						}{
							\omega \Order_{v'}
						}
		\end{equation}
	minus the number
		\begin{equation} \label{eq: T prime part of conductor formula}
				\sum_{v \in S}
					\frac{\deg(v)}{e_{v' / v}}
					\length_{\Order_{v'}}
						\frac{
							\det(\Lie(\mathcal{T}'^{0}))^{\ast} \tensor_{\Order_{S'}} \Order_{v'}
						}{
							\omega \Order_{v'}
						},
		\end{equation}
	where the length of $L / M$ for two finite free $\Order_{v'}$-modules $L \subset M$
	means the negative of the length of $M / L$.
	The number \eqref{eq: T part of conductor formula} is equal to
		\[
				\sum_{v \in S}
					\deg(v)
					\length_{\Order_{v}}
						\frac{
							\det(\Lie(\mathcal{T}^{0}))^{\ast} \otimes_{\Order_{S}} \Order_{v}
						}{
							\omega \Order_{v}
						}
			=
				- \deg(\det(\Lie(\mathcal{T}^{0}))).
		\]
	Similarly the number \eqref{eq: T prime part of conductor formula} is equal to
		\begin{align*}
			&
					\frac{[k' : k]}{[K' : K]}
					\sum_{v' \in S'}
						\deg(v')
						\length_{\Order_{v'}}
							\frac{
								\det(\Lie(\mathcal{T}'^{0}))^{\ast} \otimes_{\Order_{S'}} \Order_{v'}
							}{
								\omega \Order_{v'}
							}
			\\
			&	=
					- \frac{[k' : k]}{[K' : K]}
					\cdot
					\deg(\det(\Lie(\mathcal{T}'^{0}))),
		\end{align*}
	where the degree is relative to the field of constants $k'$ of $K'$.
	The group $\mathcal{T}'^{0}$ is a finite product of copies of $\Gm$ over $S'$.
	Hence the degree of its Lie algebra is zero.
	This proves the proposition.
\end{proof}

\begin{Prop} \label{prop: functional equation for torus determined}
	\[
			L(T, 1 - s)
		=
			q^{- \chi(S, \Lie(\mathcal{T}^{0})) (2 s - 1)} \cdot
			L(Y, s).
	\]
\end{Prop}

\begin{proof}
	This follows from \eqref{eq: functional equation for torus},
	Propositions \ref{prop: sing of functional equation}
	and \ref{prop: conductor formula to torus},
	and the Riemann-Roch formula
		\[
				\chi(S, \Lie(\mathcal{T}^{0}))
			=
				(1 - g) d + \deg(\Lie(\mathcal{T}^{0})).
		\]
\end{proof}

Recall from Corollary \ref{weighted} that
$r_{T} = - r = - \rank Y(K)$.

\begin{Thm} \label{thm: L value formula for torus}
	\[
			\lim_{s \to 1} \frac{L(T, s)}{(s - 1)^{r_T}}
		=
				\chi_{W}(S, \mathcal{T}^{0})^{-1}
			\cdot
				q^{\chi(S, \Lie \mathcal{T}^{0})}
			\cdot
				(\log q)^{r_T}.
	\]
\end{Thm}

\begin{proof}
	We have
		\[
				\chi_{W}(S, \mathcal{T}^{0})^{-1}
			=
				\chi_{W}(S, \Tilde{\mathcal{Y}})
			=
				\chi_{W}(S, \mathcal{Y})
		\]
	by Proposition \ref{prop: duality and Euler characteristic}
	and the constructibility of $R^{1} j_{\ast} Y$.
	Hence the result follows from
	Theorem \ref{thm: L value formula for Z constructible sheaf},
	Propositions \ref{prop: sign of L value for Z constructible sheaf at zero}
	and \ref{prop: functional equation for torus determined}.
\end{proof}


\section{Calculations of the Weil-\'etale Euler characteristic}

The goal of this section is to express the Weil-\'etale Euler characteristic
$\chi_{W}(S, \mathcal{T}^{0})$ 
in terms of classical invariants. Define
	\[
			\Phi_{T}(k)
		:=
			\bigoplus_{v} \pi_{0}(\mathcal{T}_{v})(k(v))
	\]
and let
	\[
		T(\Adele_{K})=\dirlim_{U\subseteq S}\;\;\prod_{v\in U}
		\mathcal{T}^{0}(\Order_{v})\times \prod_{v\not \in U}T(K_v)
	\]
be the restricted direct product of $T(K_v)$ with respect to the connected
components $\mathcal{T}^{0}(\Order_{v})$.  
For $i_{v} \colon v \into \Order_{v}$ the inclusion,
we have an exact sequence
		\begin{equation} \label{eq: connected etale over local ring for torus}
				0
			\to
				\mathcal{T}^{0}
			\to
				\mathcal{T}
			\to
				i_{v, \ast} \pi_{0}(\mathcal{T}_{v})
			\to
				0
		\end{equation}
of fppf sheaves over $\Order_{v}^{h}$.
Taking sections over $\Order_{v}$, we obtain an exact sequence
	\begin{equation}\label{ttorus}
			0
		\to
			\mathcal{T}^{0}(\Order_{v})
		\to
			T(K_{v})
		\to
			\pi_{0}(\mathcal{T}_{v})(k(v))
		\to
			0.
	\end{equation}
because $\mathcal{T}(\Order_{v})=T(K_v)$ and 
$H^{n}(\Order_{v}, \mathcal{T}^{0}) \cong H^{n}(k(v), \mathcal{T}_{v}^{0}) = 0$
for $n\geq 1$ by \cite[Chapter III, Remark 3.11 (b)]{Mil80} and Lang's theorem.
We also have an exact sequence
	\begin{equation} \label{eq: connected etale for tori, henselian points}
			0
		\to
			\mathcal{T}^{0}(\Order_{v}^{h})
		\to
			T(K_{v}^{h})
		\to
			\pi_{0}(\mathcal{T}_{v})(k(v))
		\to
			0
	\end{equation}
since $\mathcal{T}(\Order_{v}^{h})=T(K_v^{h})$ and 
$H^{n}(\Order_{v}^{h}, \mathcal{T}^{0}) \cong H^{n}(k(v), \mathcal{T}_{v}^{0}) = 0$
similarly.
Taking the restricted direct product of \eqref{ttorus}, we obtain an exact sequence
	\[
			0
		\to
			\mathcal{T}^{0}(\Order_{\Adele_{K}})
		\to
			T(\Adele_{K})
		\to
			\Phi_{T}(k)
		\to
			0.
	\]
Define $\Cl(\mathcal{T}^{0})$ to be the quotient
	\[
			\frac{T(\Adele_{K})}{T(K) + \mathcal{T}^{0}(\Order_{\Adele_{K}})}
		\cong
			\frac{\Phi_{T}(k)}{T(K)},
	\]
and let 
\[
\Sha(T)=\Ker \left( H^1(K,T)\to \prod H^1(K_v,T) \right)
\]	
be the Tate-Shafarevich group of $T$. We note that the Tate-Shafarevich
group does not change if we use $H^1(K_v^h,T)$ instead of $H^1(K_v,T)$
since $H^1(K_v^h,T) \cong H^1(K_v,T)$.

\begin{Prop} \label{prop: identifying first Weil etale}
	There exist a canonical exact sequence and a canonical isomorphism
		\begin{gather*}
					0
				\to
					\Cl(\mathcal{T}^{0})
				\to
					H_{W}^{1}(S, \mathcal{T}^{0})
				\to
					\Sha(T)
				\to
					0,
			\\
					H_{W}^{1}(S, \Tilde{\mathcal{Y}})_{\tor}
				\cong
					H^{1}(K, Y).
		\end{gather*}
\end{Prop}

\begin{proof}
	The exact sequence
		\[
				0
			\to
				H^{1}(S, \mathcal{T}^{0})
			\to
				H_{W}^{1}(S, \mathcal{T}^{0})
			\to
				\mathcal{T}^{0}(S) \tensor \Q
		\]
	and the finiteness of $\mathcal{T}^{0}(S)$
	(Proposition \ref{prop: finiteness and freeness for tori and lattices})
	shows $H^{1}(S, \mathcal{T}^{0})\cong H_{W}^{1}(S, \mathcal{T}^{0})$.
	The localization sequence gives an exact sequence
		\[
				T(K)
			\to
				\bigoplus_{v}
					H_{v}^{1}(\Order_{v}^{h}, \mathcal{T}^{0})
			\to
				H^{1}(S, \mathcal{T}^{0})
			\to
				H^{1}(K, T)
			\to
				\bigoplus_{v}
					H_{v}^{2}(\Order_{v}^{h}, \mathcal{T}^{0}).
		\]
	Now consider the analogous sequence for $\Spec \Order_{v}^{h}$. 
	The vanishing of $H^{n}(\Order_{v}^{h}, \mathcal{T}^{0})$ for
	$n\geq 1$ gives an isomorphism
	$H^{1}(K_{v}^{h}, T)\cong H_{v}^{2}(\Order_{v}^{h}, \mathcal{T}^{0})$
	as well as a short exact sequence
		\[
				0
			\to
				\mathcal{T}^{0}(\Order_{v}^{h})
			\to
				T(K_{v}^{h})
			\to
			H_{v}^{1}(\Order_{v}^{h}, \mathcal{T}^{0})
			\to
				0,
		\]
	which implies $H_{v}^{1}(\Order_{v}^{h}, \mathcal{T}^{0})\cong 
	\pi_{0}(\mathcal{T}_{v})(k(v))$ by comparing to 
	\eqref{eq: connected etale for tori, henselian points}.
	
	For the isomorphism, we use that
	$H^{1}(S, \Tilde{\mathcal{Y}})_{\tor}
	\isomto H_{W}^{1}(S, \Tilde{\mathcal{Y}})_{\tor}$
	as above. Hence it is enough show that the natural map
	$H^{1}(S, \Tilde{\mathcal{Y}}) \to H^{1}(K, Y)$ is an isomorphism,
	which follows from $H_{v}^{n}(\Order_{v}^{h}, \Tilde{\mathcal{Y}}) = 0$
	for $n \le 2$. For this, it is enough to show that
	$H_{v}^{n}(\Order_{v}^{sh}, \Tilde{\mathcal{Y}}) = 0$
	for $n \le 2$ which follows because
		\[
				H^{n}(\Order_{v}^{sh}, \Tilde{\mathcal{Y}})
			\cong
				\begin{cases}
						H^{n}(K_{v}^{sh}, Y)
					&
						n = 0,1;
					\\
						0
					&
						n \ne 0, 1.
				\end{cases}
		\]
by definition of $\Tilde{\mathcal{Y}}$.
\end{proof}

Denote the first map in the exact sequence
in Proposition \ref{prop: identifying first Weil etale}
by $\cl_{T}$:
	\[
			\cl_{T}
		\colon
			\Cl(\mathcal{T}^{0})
		\into
			H_{W}^{1}(S, \mathcal{T}^{0}).
	\]
As seen in the proof of Proposition \ref{prop: identifying first Weil etale}, 
$\cl_{T}$ is the composite of the natural maps
	\[
			\frac{T(\Adele_{K})}{T(K) + \mathcal{T}^{0}(\Order_{\Adele_{K}})}
		\isomto
			\frac{
				\bigoplus_{v}
					H_{v}^{1}(\Order_{v}^{h}, \mathcal{T}^{0})
			}{
				T(K)
			}
		\into
			H^{1}(S, \mathcal{T}^{0})
		\isomto
			H_{W}^{1}(S, \mathcal{T}^{0}).
	\]
We denote the map induced on the torsion-free quotients by $\cl_{T} / \tor$:
	\begin{gather*}
				\cl_{T} / \tor
			\colon
				\Cl(\mathcal{T}^{0}) / \tor
			\into
				H_{W}^{1}(S, \mathcal{T}^{0})  / \tor.
	\end{gather*}
By the functoriality of class groups $\Cl$ 
and connected N\'eron models, we have a natural map
	\begin{equation} \label{eq: Cl times Y functoriality pairing}
			Y(K)
		=
			\Hom_{K}(T, \Gm)
		\cong
			\Hom_{S}(\mathcal{T}^{0}, \Gm)
		\to
			\Hom(\Cl(\mathcal{T}^{0}), \Cl(\Gm)).
	\end{equation}
Hence we have a canonical pairing
	\begin{equation} \label{eq: height pairing for tori}
			h_{T}
		\colon
			\Cl(\mathcal{T}^{0}) \times Y(K)
		\to
			\Cl(\Gm)
		=
			\Pic(S)
		\stackrel{\deg}{\to}
			\Z.
	\end{equation}
Recall the isomorphism $Y(K) \cong H_{W}^{0}(S, \Tilde{\mathcal{Y}})$
and the morphism $e \colon H_{W}^{0}(S, \Tilde{\mathcal{Y}}) \to H_{W}^{1}(S, \Tilde{\mathcal{Y}})$.
We denote the composite
$Y(K) \to H_{W}^{1}(S, \Tilde{\mathcal{Y}})$ by $e$ by abuse of notation.

\begin{Prop} \label{prop: regulator comparison for tori}
	The composite
		\[
				\Cl(\mathcal{T}^{0}) \times Y(K)
			\stackrel{\cl_{T} \times e}{\longrightarrow}
				H_{W}^{1}(S, \mathcal{T}^{0}) \times H_{W}^{1}(S, \Tilde{\mathcal{Y}})
			\to
				\Z,
		\]
	where the last (perfect) pairing is the pairing of 
	Corollary \ref{cor: global duality for tori, each term},
	agrees with $h_{T}$.
\end{Prop}

\begin{proof}
	The pairing $\mathcal{T}^{0} \tensor^{L} \Tilde{\mathcal{Y}} \to \Gm$ over $S$ in \eqref{ppaaiirr}
	induces a commutative diagram
		\[
			\begin{CD}
					T(K) \times Y(K)
				@>>>
					K^{\times}
				\\ @VVV @VVV \\
					\bigoplus_{v} H_{v}^{1}(\Order_{v}^{h}, \mathcal{T}^{0}) \times Y(K)
				@> \cup >>
					\bigoplus_{v} H_{v}^{1}(\Order_{v}^{h}, \Gm),
			\end{CD}
		\]
	where the vertical maps are coboundary maps of localization sequences.
	On the cokernels of the vertical maps, this diagram induces a pairing
		\[
				\Cl(\mathcal{T}^{0}) \times Y(K)
			\to
				\Cl(\Gm).
		\]
	This agrees with the map \eqref{eq: Cl times Y functoriality pairing}.
	Hence the naturality of the cup product shows that the diagram
		\[
			\begin{CD}
					\Cl(\mathcal{T}^{0}) \times Y(K)
				@>>>
					\Cl(\Gm) = \Pic(S)
				@>>>
					\Z
				\\ @VV \cl_T \times \id V @| @| \\
					H_{W}^{1}(S, \mathcal{T}^{0}) \times Y(K)
				@> \cup >>
					H_{W}^{1}(S, \Gm) = \Pic(S)
				@>>>
					\Z
			\end{CD}
		\]
	is commutative, where the upper horizontal pairing is $h_{T}$.
	Also consider the commutative diagram
		\[
			\begin{CD}
					H_{W}^{1}(S, \mathcal{T}^{0}) \times Y(K)
				@> \cup >>
					H_{W}^{1}(S, \Gm) = \Pic(S)
				@>>>
					\Z
				\\ @VV \id \times e V @VV e V @| \\
					H_{W}^{1}(S, \mathcal{T}^{0}) \times H_{W}^{1}(S, \Tilde{\mathcal{Y}})
				@> \cup >>
					H_{W}^{2}(S, \Gm)
				@=
					\Z,
			\end{CD}
		\]
	where the commutativity of the right square follows from the 
	geometric connectivity of $S$ over $k$. Combining these two diagrams,
	we get the result.
\end{proof}

With these preparations we can determine the Euler characteristic 
of the torus.

\begin{Prop} \label{prop: Weil etale Euler char for torus, free part}
	\[
			\chi(H_{W}^{\ast}(S, \mathcal{T}^{0}) / \tor, e)^{-1}
		=
			\frac{\# \Coker(\cl_{T} / \tor)}{\Disc(h_{T})}.
	\]
\end{Prop}

\begin{proof}
	The only non-trivial map for this Euler characteristic is
		\[
				e
			\colon
				H_{W}^{1}(S, \mathcal{T}^{0}) / \tor
			\to
				H_{W}^{2}(S, \mathcal{T}^{0}) / \tor
		\]
	by Proposition \ref{prop: finiteness and freeness for tori and lattices}.
	Its linear dual is
		\[
				e
			\colon
				Y(K)
			\to
				H_{W}^{1}(S, \Tilde{\mathcal{Y}}) / \tor
		\]
	by Corollary \ref{cor: global duality for tori, each term}
	and the proof of Proposition \ref{prop: duality and Euler characteristic}.
	Hence Proposition \ref{prop: regulator comparison for tori} gives the result.
\end{proof}

\begin{Prop} \label{prop: Weil etale Euler char for torus, torsion part}
	Denote the alternating product of the orders of $H_{W}^{\ast}(S, \mathcal{T}^{0})_{\tor}$
	by $\chi(H_{W}^{\ast}(S, \mathcal{T}^{0})_{\tor})$.
	Then
		\[
				\chi(H_{W}^{\ast}(S, \mathcal{T}^{0})_{\tor})^{-1}
			=
				\frac{
					\# \Cl(\mathcal{T}^{0})_{\tor} \cdot \# \Sha(T)
				}{
					\# \mathcal{T}^{0}(S) \cdot \# H^{1}(K, Y) \cdot \# \Coker(\cl_{T} / \tor)
				}.
		\]
\end{Prop}

\begin{proof}
	We have an exact sequence of finite groups
		\[
					0
				\to
					\Cl(\mathcal{T}^{0})_{\tor}
				\stackrel{\cl_{T}}{\longrightarrow}
					H_{W}^{1}(S, \mathcal{T}^{0})_{\tor}
				\to
					\Sha(T)
				\to
					\Coker(\cl_{T} / \tor)
				\to
					0
		\]
	by Proposition \ref{prop: identifying first Weil etale}.
	Also
		\[
				\# H_{W}^{2}(S, \mathcal{T}^{0})_{\tor}
			=
				\# H_{W}^{1}(S, \Tilde{\mathcal{Y}})_{\tor}
			=
				\# H^{1}(K, Y)
		\]
	by Corollary \ref{cor: global duality for tori, each term}
	and Proposition \ref{prop: identifying first Weil etale}.
	Therefore
		\begin{align*}
					\chi(H_{W}^{\ast}(S, \mathcal{T}^{0})_{\tor})^{-1}
			&	=
					\frac{
						\# H_{W}^{1}(S, \mathcal{T}^{0})_{\tor}
					}{
						\# H_{W}^{0}(S, \mathcal{T}^{0})_{\tor} \cdot \# H_{W}^{2}(S, \mathcal{T}^{0})_{\tor}
					}
			\\
			&	=
					\frac{
						\# \Cl(\mathcal{T}^{0})_{\tor} \cdot \# \Sha(T)
					}{
						\# \mathcal{T}^{0}(S) \cdot \# H^{1}(K, Y) \cdot \# \Coker(\cl_{T} / \tor)
					}.
		\end{align*}
\end{proof}

\begin{Prop} \label{prop: Weil etale Euler char for torus}
	\[
			\chi_{W}(S, \mathcal{T}^{0})^{-1}
		=
				\frac{
					\# \Cl(\mathcal{T}^{0})_{\tor} \cdot \# \Sha(T)
				}{
					\# \mathcal{T}^{0}(S) \cdot \# H^{1}(K, Y) \cdot \Disc(h_{T})
				}.
	\]
\end{Prop}

\begin{proof}
	The number $\chi_{W}(S, \mathcal{T}^{0})$ is the product of
	$\chi(H_{W}^{\ast}(S, \mathcal{T}^{0}) / \tor, e)$ and
	$\chi(H_{W}^{\ast}(S, \mathcal{T}^{0})_{\tor})$
	by the proof of \cite[Theorem 9.1]{Gei04}.
	Therefore Propositions \ref{prop: Weil etale Euler char for torus, free part}
	and \ref{prop: Weil etale Euler char for torus, torsion part}
	give the result.
\end{proof}


\section{A Weil-\'etale Tamagawa number formula for tori}
\label{sec: Ono Tamagawa number formula in terms of Weil etale cohomology}

In this section, we express the Tamagawa number of $T$,
defined by Ono in \cite{Ono61} and redefined (in the function field case) 
by Oesterl\'e \cite{Oes84}, 
in terms of arithmetic-geometric invariants defined without using Haar measures.
We use this to reprove Ono-Oesterl\'e's Tamagawa number 
formula \cite{Ono63}, \cite{Oes84}.

We begin by recalling the Tamagawa number from
\cite[Sections 3.1--3.5]{Ono61}, \cite[Chapter I]{Oes84}.
Set
	\[
		d = \dim(T), \quad r = \rank(Y(K)), \quad g= \text{genus}\; S.
	\]
The function $L(T, s)$ has a pole of order $r$ at $s = 1$
by Theorem \ref{thm: L value formula for torus}.
We define
	\[
			\rho(T)
		=
			\lim_{s \to 1}
				L(T, s) (s - 1)^{r}.
	\]
We denote the sum of the valuations
$\Adele_{K}^{\times} \onto \bigoplus_{v} \Z \onto \Z$
by $\deg$. A character $\chi \in Y(K) = \Hom_{K}(T, \Gm)$
induces a homomorphism $\bar \chi: T(\Adele_{K}) \to \Adele_{K}^{\times}$,
and as in \cite[Section 3.1]{Ono61} we define the subgroup
$$T(\Adele_{K})^{1} = \{ x\in T(\Adele_{K}) \mid \deg(\bar \chi(x)) = 0
\;\forall \chi \in Y(K)\}.$$
In other words, $T(\Adele_{K})^{1}$ is the inverse image
of the left kernel of the pairing $h_{T}$ in \eqref{eq: height pairing for tori}
under the quotient map $T(\Adele_{K}) \onto \Cl(\mathcal{T}^{0})$.
This group contains $T(K)$ and $\mathcal{T}^{0}(\Order_{\Adele_{K}})$.
Recall that the quotient $\Cl(\mathcal{T}^{0})$ of
$T(\Adele_{K})$ by $T(K) + \mathcal{T}^{0}(\Order_{\Adele_{K}})$
is finitely generated by
Propositions \ref{prop: identifying first Weil etale}
and \ref{prop: cohom of torus is finitely generated and bounded}.

\begin{Prop} \label{prop: describe the degree zero adelic points}
	In the exact sequence
		\[
				0
			\to
				\frac{
					T(\Adele_{K})^{1}
				}{
					T(K) + \mathcal{T}^{0}(\Order_{\Adele_{K}})
				}
			\to
				\Cl(\mathcal{T}^{0})
			\to
				\frac{T(\Adele_{K})}{T(\Adele_{K})^{1}}
			\to
				0,
		\]
	the first term is the torsion part of $\Cl(\mathcal{T}^{0})$.
	In particular, $T(\Adele_{K})^{1}$ is the inverse image of 
	$\Cl(\mathcal{T}^{0})_{\tor}$
	under the surjection $T(\Adele_{K}) \onto \Cl(\mathcal{T}^{0})$.
\end{Prop}

\begin{proof}
	The quotient $T(\Adele_{K})^{1} / T(K)$ is compact
	by \cite[Theorem 3.1.1]{Ono61}.
	Hence the first term is finite.
	The third term is torsion-free by definition.
\end{proof}

Let $\omega$ be a non-zero invariant differential form on $T / K$
of maximal degree.
It induces a canonical Haar measure on $\Lie(T)(K_{v})$ and $T(K_{v})$ for each $v$,
which we denote by $\mu_{v}$.
Set
	\[
			P_{v}(T, t)
		=
			\det(1 - \varphi_{v} t \mid V_{l}(T)(-1)^{I_{v}}).
	\]
We define the Tamagawa measure $\mu_{T(\Adele_{K})}$ on $T(\Adele_{K})$ by
	\[
			\mu_{T(\Adele_{K})}
		=
			\frac{1}{\rho(T) \cdot q^{(g - 1) d}}
			\prod_{v} P_{v}(T, N(v)^{-1})^{-1} \mu_{v}.
	\]
The infinite product evaluated on open compact subgroups absolutely converges
\cite[Section 3.3]{Ono61}
and hence defines a Haar measure on $T(\Adele_{K})$, which
does not depend on the choice of $\omega$ by the product formula.
The composite map
	\[
			T(\Adele_{K})
		\onto
			\Cl(\mathcal{T}^{0})
		\stackrel{h_{T}}{\longrightarrow}
			\Hom(Y(K), \Z)
		\stackrel{n\mapsto q^n}{\longrightarrow}
			\Hom(Y(K), \R_{>0}^{\times})
	\]
is denoted by $\vartheta$ in \cite[Chapter I, Section 5.5]{Oes84}.
Recall from \cite[Section 3.5]{Ono61} and \cite[Chapter I, Definition 5.12]{Oes84} that
the Tamagawa number $\tau(T)$ is defined as
	\[
			\tau(T)
		=
			\frac{
				\mu_{T(\Adele_{K})} \bigl(
					T(\Adele_{K})^{1} / T(K)
				\bigr)
			}{
				(\log q)^{r} \cdot \Disc(h_{T})
			}
	\]
The correction factor $\Disc(h_{T})$
was introduced by Oesterl\'e \cite[Chapter I, Definition 5.9 (b)]{Oes84}
and did not appear in \cite{Ono61}.
It can be non-trivial \cite[Chapter I, Remark 5.7]{Oes84}.

\begin{Prop} \label{prop: Tamagawa number in modern terms}
	\[
			\tau(T)
		=
			\frac{
				\# \Cl(\mathcal{T}^{0})_{\tor} \cdot q^{\chi(S, \Lie \mathcal{T}^{0})}
			}{
				\# \mathcal{T}^{0}(S) \cdot \rho(T) \cdot (\log q)^{r} \cdot \Disc(h_{T})
			}.
	\]
\end{Prop}

\begin{proof}
	By Proposition \ref{prop: describe the degree zero adelic points},
	we have an exact sequence
		\[
				0
			\to
				\frac{\mathcal{T}^{0}(\Order_{\Adele_{K}})}{\mathcal{T}^{0}(S)}
			\to
				\frac{T(\Adele_{K})^{1}}{T(K)}
			\to
				\Cl(\mathcal{T}^{0})_{\tor}
			\to
				0.
		\]
	Hence $\tau(T) \cdot (\log q)^{r} \cdot \Disc(h_{T})=
	\mu_{T(\Adele_{K})} \bigl(T(\Adele_{K})^{1} / T(K)	\bigr)$
	can be written as
		\begin{align*}
			&	
						\frac{
							\# \Cl(\mathcal{T}^{0})_{\tor}
						}{
							\# \mathcal{T}^{0}(S)
						}
					\cdot
						\mu_{T(\Adele_{K})} \bigl(
							\mathcal{T}^{0}(\Order_{\Adele_{K}})
						\bigr)
			\\
			&	=
						\frac{
							\# \Cl(\mathcal{T}^{0})_{\tor}
						}{
							\# \mathcal{T}^{0}(S) \cdot \rho(T) \cdot q^{(g - 1) d}
						}
					\cdot
						\prod_{v}
							P_{v}(T, N(v)^{-1})^{-1}
							\mu_{v}(\mathcal{T}^{0}(\Order_{v})).
		\end{align*}
	To calculate the factors in the product term, we have
		\[
				P_{v}(T, N(v)^{-1})
			=
				\frac{\# \mathcal{T}^{0}(k(v))}{N(v)^{d}}
		\]
	(see the proof of \cite[Proposition 4.1]{GS20} for example)
	and
		\begin{align*}
					\mu_{v}(\mathcal{T}^{0}(\Order_{v}))
			&	=
						\# \mathcal{T}^{0}(k(v))
					\cdot
						\mu_{v}(\mathcal{T}^{0}(\ideal{p}_{v}))
			\\
			&	=
						\# \mathcal{T}^{0}(k(v))
					\cdot
						\mu_{v}(\Lie(\mathcal{T}^{0})(\ideal{p}_{v}))
			\\
			&	=
						\frac{\# \mathcal{T}^{0}(k(v))}{N(v)^{d}}
					\cdot
						\mu_{v}(\Lie(\mathcal{T}^{0})(\Order_{v})),
		\end{align*}
	where $\mathcal{T}^{0}(\ideal{p}_{v})$ denotes
	the kernel of the reduction map
	$\mathcal{T}^{0}(\Order_{v}) \onto \mathcal{T}^{0}(k(v))$
	and $\Lie(\mathcal{T}^{0})(\ideal{p}_{v})$ similarly.
	Hence
		\begin{align*}
					P_{v}(T, N(v)^{-1})^{-1}
					\mu_{v}(\mathcal{T}^{0}(\Order_{v}))
			&	=
					\mu_{v} \bigl(
						\Lie(\mathcal{T}^{0})(\Order_{v})
					\bigr)
			\\
			&	=
					N(v)^{- v(\omega)},
		\end{align*}
	where $v(\omega)$ is the order of zero at $v$ of $\omega$
	as a rational section of $\det(\Lie(\mathcal{T}^{0}))^{\ast}$.
	Therefore
		\begin{align*}
					\prod_{v}
						P_{v}(T, N(v)^{-1})^{-1}
						\mu_{v}(\mathcal{T}^{0}(\Order_{v}))
			&	=
					\prod_{v}
						q^{- \deg(v) \cdot v(\omega)}
			\\
			&	=
					q^{\deg(\Lie(\mathcal{T}^{0}))}
			\\
			&	=
					q^{\chi(S, \Lie \mathcal{T}^{0}) - (1 - g) d},
		\end{align*}
	where the last equality is the Riemann-Roch theorem.
\end{proof}

\begin{Prop}
	\[
			\tau(T)
		=
			\frac{\# H^{1}(K, Y)}{\# \Sha(T)}.
	\]
\end{Prop}

\begin{proof}
	This follows from Theorem \ref{thm: L value formula for torus},
	Propositions \ref{prop: Weil etale Euler char for torus}
	and \ref{prop: Tamagawa number in modern terms}.
\end{proof}

This reproves Ono-Oesterl\'e's Tamagawa number formula
\cite[Section 5, Main theorem]{Ono63}, \cite[Chapter IV, Corollary 3.3]{Oes84}.


\section{$1$-motives}
\label{sec: 1 motives}

Let $M$ be a $1$-motive over $K$. The goal of this section is to 
combine the results of this paper for tori and lattices with the
results of for abelian varieties of \cite{GS20} to obtain a formula for the
$L$-function of $M$ at $s = 1$. More precisely, we define a model 
$\mathcal{M}^{\Delta}$ over $S$ such that,
assuming the finiteness of the Tate-Shafarevich group of the abelian 
variety component of $M$, the groups 
$H_{W}^{\ast}(S, \mathcal{M}^{\Delta})$ are finitely generated, 
and the leading coefficient of the $L$-function of $M$ at $s = 1$
can be expressed in terms of $H_{W}^{\ast}(S, \mathcal{M}^{\Delta})$,

Let $M= [X \to G]$ be a $1$-motive over $K$,
where $X$ and $G$ are a lattice and a semi-abelian variety over $K$
placed in degree $-1$ and $0$, respectively.
Let $T$ be the torus part of $G$ and $A$ the abelian variety quotient of $G$.
Let $Y$ be the character lattice of $T$.
Denote $\mathcal{X} = j_{\ast} X$,
where $j \colon \Spec K \into S$.
Denote the N\'eron and connected N\'eron model of $G$ over $S$  by
$\mathcal{G}$ and $\mathcal{G}^{0}$, respectively.
Define $\mathcal{X}^{\Delta}$ by the fiber product
	\[
		\begin{CD}
				\mathcal{X}^{\Delta} @>>> \mathcal{G}^{0}
				\\ @VVV @VVV \\
				\mathcal{X} @>>> \mathcal{G}.
		\end{CD}
	\]
Note that despite the notation,
the scheme $\mathcal{X}^{\Delta}$ depends not only on $\mathcal{X}$
but the whole data $M = [X \to G]$.
The fiber $\mathcal{X}_{v}^{\Delta}$ of $\mathcal{X}^{\Delta}$ at a place $v$
is the kernel of the morphism
$\mathcal{X}_{v} \to \pi_{0}(\mathcal{G}_{v})$ over $k(v)$.
Therefore the group of geometric points of $\mathcal{X}_{v}^{\Delta}$ is
the kernel of the map
$X(K_{v}^{sh}) \to \pi_{0}(\mathcal{G}_{v})(\closure{k(v)})$.

\begin{Prop}
	The scheme $\mathcal{X}^{\Delta}$ is a $\Z$-constructible 
	\'etale subsheaf of $\mathcal{X}$
	such that the quotient $\mathcal{X} / \mathcal{X}^{\Delta}$ is supported
	on finitely many closed points of $S$.
\end{Prop}

\begin{proof}
	It is enough to show that the morphism
	$\mathcal{X}_{v} \to \pi_{0}(\mathcal{G}_{v})$ over $k(v)$
	is zero for almost all places $v$.
	We have an exact sequence
		\[
				0
			\to
				\mathcal{G}^{0}
			\to
				\mathcal{G}
			\to
				\bigoplus_{v}
					i_{v, \ast} \pi_{0}(\mathcal{G}_{v})
			\to
				0
		\]
	on $S_{\et}$.
	We want to show that the induced morphism
	$\mathcal{X} \to \bigoplus_{v} i_{v, \ast} \pi_{0}(\mathcal{G}_{v})$
	factors through a finite partial direct sum.
	Let $U \subseteq S$ be a dense open subscheme
	such that $\mathcal{X}$ is a lattice over $U$.
	Let $U' / U$ be a finite \'etale Galois covering with Galois group $G$
	trivializing $\mathcal{X}$.
	Then the morphism
	$\mathcal{X} \times_{S} U \to \bigoplus_{v \notin U} i_{v, \ast} \pi_{0}(\mathcal{G}_{v})$
	over $U$ corresponds to a $G$-module homomorphism
	$\mathcal{X}(U') \to \bigoplus_{v \notin U} \pi_{0}(\mathcal{G}_{v})(U' \times_{U} k(v))$.
	Since $\mathcal{X}(U')$ is finitely generated,
	this homomorphism indeed factors through a finite partial direct sum.
\end{proof}

\begin{Ex} \label{ex: delta Neron for lattice part} \mbox{}
	\begin{enumerate}
		\item \label{item: rational function example}
			Take $X = \Z$ and $G = \Gm$,
			so that the morphism $X \to G$ corresponds to a non-zero 
			rational function $f \in K^{\times}$.
			Then for any place $v$, the map
			$X(K_{v}^{sh}) \to \pi_{0}(\mathcal{G}_{v})(\closure{k(v)})$
			is the map $\Z \to \Z$ given by multiplication by the 
			valuation $v(f)$ of $f$.
			Therefore $\mathcal{X}_{v}^{\Delta} \ne \mathcal{X}_{v} = \Z$ 
			if and only if
			$f$ has a zero or pole at $v$,
			in which case $\mathcal{X}_{v}^{\Delta} = 0$.
			Therefore
				\[
						\mathcal{X} / \mathcal{X}^{\Delta}
					=
						\bigoplus_{v \in \Sigma}
							i_{v, \ast} \Z,
				\]
			where $\Sigma$ is the set of places
			where $f$ has a zero or pole.
		\item \label{item: rational point example}
			Take $X = \Z$, and assume $G = A$.
			Then the morphism $X \to G$ corresponds to a rational point 
			$a \in A(K)$.
			For any place $v$, the map
			$X(K_{v}^{sh}) \to \pi_{0}(\mathcal{G}_{v})(\closure{k(v)})$
			corresponds to the image 
			$a_{v} \in \pi_{0}(\mathcal{A}_{v})(k(v))$ of $a$.
			Note that $\pi_{0}(\mathcal{A}_{v})(k(v))$ is a finite group.
			Therefore $\mathcal{X}_{v}^{\Delta} \subseteq \mathcal{X}_{v}$
			is given by the finite index subgroup $n_{v} \Z \subset \Z$,
			where $n_{v}$ is the order of $a_{v}$.
			We have $n_{v} = 1$ for almost all $v$
			since $A$ has good reduction almost everywhere
			and so $\pi_{0}(\mathcal{A}_{v}) = 0$ for almost all $v$.
			Thus
				\[
						\mathcal{X} / \mathcal{X}^{\Delta}
					=
						\bigoplus_{v}
							i_{v, \ast}(\Z / n_{v} \Z)
				\]
			is constructible in this case.
	\end{enumerate}
\end{Ex}

Define $\mathcal{M}^{\Delta}$ to be the complex
	\[
			\mathcal{M}^{\Delta}
		:=
			[\mathcal{X}^{\Delta} \to \mathcal{G}^{0}].
	\]
The $l$-adic representation $V_{l}(M)$ over $K$ associated with $M$
fits in the exact sequence
	\[
			0
		\to
			V_{l}(G)
		\to
			V_{l}(M)
		\to
			X \tensor \Q_{l}
		\to
			0.
	\]
Define $L(M, s)$ to be the Hasse-Weil $L$-function of
$V_{l}(M)(-1)$.

\begin{Ex}
	The sheaf $\mathcal{X} / \mathcal{X}^{\Delta}$ non-trivially contributes to
	$L(M, s)$ and $\chi_{W}(S, \mathcal{M}^{\Delta})$ in general.
	To see this, first we have
		\[
				L(X, s)
			=
					L(\mathcal{X}^{\Delta}, s)
				\cdot
					L(\mathcal{X} / \mathcal{X}^{\Delta}, s),
		\]
		\[
				\chi_{W}(S, \mathcal{X})
			=
					\chi_{W}(S, \mathcal{X}^{\Delta})
				\cdot
					\chi_{W}(S, \mathcal{X} / \mathcal{X}^{\Delta}).
		\]
	In Example \ref{ex: delta Neron for lattice part}
	\eqref{item: rational function example},
	we have
		\[
				L(\mathcal{X} / \mathcal{X}^{\Delta}, s)
			=
				\prod_{v \in \Sigma}
					(1 - N(v)^{-s})^{-1}.
		\]
	This has a pole of order $\# \Sigma$ at $s = 0$.
	We have
		\[
				\lim_{s \to 0}
					L(\mathcal{X} / \mathcal{X}^{\Delta}, s) \cdot s^{\# \Sigma}
			=
				\prod_{v \in \Sigma}
					(\log N(v))^{-1}
			=
				\prod_{v \in \Sigma}
					(\deg(v))^{-1}
				\cdot
				(\log q)^{- \# \Sigma}.
		\]
	Also
		\[
				\chi_{W}(S, \mathcal{X} / \mathcal{X}^{\Delta})
			=
				\prod_{v \in \Sigma}
					(\deg(v))^{-1}.
		\]
	This is consistent with Theorem \ref{thm: L value formula for Z constructible sheaf}.
\end{Ex}

\begin{Prop} \label{prop: sequence of connected Neron for semiabelian}
	The complex of sheaves on $S_\et$ 
		\begin{equation} \label{eq: sequence of connected Neron for semiabelian}
				0
			\to
				\mathcal{T}^{0}
			\to
				\mathcal{G}^{0}
			\to
				\mathcal{A}^{0}
			\to
				0
		\end{equation}
	is exact at $\mathcal{T}^{0}$ and $\mathcal{A}^{0}$, and its 
	cohomology at $\mathcal{G}^{0}$ is an \'etale skyscraper sheaf with 
	finite stalks.
\end{Prop}

Note that the exactness of the sequence is
not considered in the category of group schemes over $S$. For example, 
the morphism $\mathcal{T}^{0} \to \mathcal{G}^{0}$ may not be a closed 
immersion as noted in \cite[Remark 4.8 (b)]{Cha00}.

\begin{proof}
	For any $v \in S$, the stalk of $R^{1} j_{\ast} T$ at $\closure{v} = \Spec \closure{k(v)}$ is
	$H^{1}(K_{v}^{sh}, T)$,
	which is zero by \cite[Chapter X, Section 7, ``Application'']{Ser79}.
	Hence we have an exact sequence
	$0 \to \mathcal{T} \to \mathcal{G} \to \mathcal{A} \to 0$
	in $\Ab(S_{\et})$.
	In particular, the morphism $\mathcal{T}^{0} \to \mathcal{G}^{0}$ is injective in $\Ab(S_{\et})$.
	For almost all $v$, the sequence
	\eqref{eq: sequence of connected Neron for semiabelian}
	pulled back to $\Order_{v}$ is exact.
	Let $C_v$ and $D_v$ the cohomology of the complex 
		\[
				0
			\to
				\mathcal{T}^{0}(\Order_{v}^{sh})
			\to
				\mathcal{G}^{0}(\Order_{v}^{sh})
			\to
				\mathcal{A}^{0}(\Order_{v}^{sh})
			\to
				0
		\]
    in the middle and on the right, respectively. 
	It suffices to show that $C_v$ is finite and $D_v$ is zero.
	It follows from the diagram
		\[
			\begin{CD}
				@. 0 @. 0 @. 0 @.
				\\ @. @VVV @VVV @VVV @. \\
					0
				@>>>
					\mathcal{T}^{0}(\Order_{v}^{sh})
				@>>>
					\mathcal{G}^{0}(\Order_{v}^{sh})
				@>>>
					\mathcal{A}^{0}(\Order_{v}^{sh})
				@>>>
					0
				\\ @. @VVV @VVV @VVV \\
					0
				@>>>
					\mathcal{T}(\Order_{v}^{sh})
				@>>>
					\mathcal{G}(\Order_{v}^{sh})
				@>>>
					\mathcal{A}(\Order_{v}^{sh})
				@>>>
					0
				\\ @. @VVV @VVV @VVV \\
					0
				@>>>
					\pi_{0}(\mathcal{T}_{v})(\closure{k(v)})
				@>>>
					\pi_{0}(\mathcal{G}_{v})(\closure{k(v)})
				@>>>
					\pi_{0}(\mathcal{A}_{v})(\closure{k(v)})
				@>>>
					0
				\\ @. @VVV @VVV @VVV @. \\
				@. 0 @. 0 @. 0, @.
			\end{CD}
		\]
	with exact columns and exact middle row that 
		\[
				C_v
			\cong
				\Ker \Bigl(
					\pi_{0}(\mathcal{T}_{v})(\closure{k(v)}) \to \pi_{0}(\mathcal{G}_{v})(\closure{k(v)})
				\Bigr),
		\]
		\[
				D_v
			\cong
				H \Bigl(
						\pi_{0}(\mathcal{T}_{v})(\closure{k(v)})
					\to
						\pi_{0}(\mathcal{G}_{v})(\closure{k(v)})
					\to
						\pi_{0}(\mathcal{A}_{v})(\closure{k(v)})
				\Bigr).
		\]
	The groups $\pi_{0}(\mathcal{T}_{v})(\closure{k(v)})$ and
	$\pi_{0}(\mathcal{G}_{v})(\closure{k(v)})$ are finitely generated
	by \cite[Proposition 3.5]{HN11} for example.
	Therefore $C_v$ and $D_v$ are finitely generated, and do not change  
	if $\Order_{v}^{sh}$ is replaced by
	the maximal unramified extension $\Order_{v}^{\ur}$ of $\Order_{v}$ or
	its completion $\Hat{\Order}_{v}^{\ur}$.
	The group $\mathcal{G}^{0}(\Order_{v}^{\ur})$ is the union of
	profinite subgroups $\mathcal{G}^{0}(\Order_{L})$,
	where $L$ runs through finite subextensions of $K_{v}^{\ur} / K_{v}$.
	Similar statements hold for $\mathcal{T}^{0}(\Order_{v}^{\ur})$
	and $\mathcal{A}^{0}(\Order_{v}^{\ur})$.
	Hence $C_v$ and $D_v$ are unions of profinite subgroups and thus 
	they are finite.
	
	It remains to prove that $D_v = 0$.
	Let $n = \# \Aut(D_v)$.
	Then $\varphi_{v}^{n}$ acts trivially on
	$D_v = \mathcal{A}^{0}(\Hat{\Order}_{v}^{\ur}) / 
	\mathcal{G}^{0}(\Hat{\Order}_{v}^{\ur})$, so that 
	$\varphi_{v}^{n} - 1$
	is the zero map on $D_v$. 
	But $\varphi_{v}^{n} - 1$ is surjective on 
	$\mathcal{A}^{0}(\Hat{\Order}_{v}^{\ur})$ by Lang's theorem.
\end{proof}

Recall that $Y$ denotes the character lattice of $T$.

\begin{Cor} \label{prop: decomposition of Q rank} \mbox{}
	\begin{enumerate}
		\item
			The groups 
			$H^{\ast}_W(S, \mathcal{M}^{\Delta})$ have finite ranks
			(namely, they become finite-dimensional after $\tensor \Q$)
			and we have
				\begin{align*}
							r_{M}
					&	:=
							-\sum_{i}
							(-1)^{i}\cdot i \cdot\rank H^{i}_W(S, \mathcal{M}^{\Delta})
					\\
					&	=
							r_{T} + r_{A} - r_{\mathcal{X}^{\Delta}}
					\\
					&	=
							\rank A(K) - \rank X(K) - \rank Y(K)
					\\
					&	\qquad
							+ \sum_{v}
								\rank (\mathcal{X} / \mathcal{X}^{\Delta})(k(v)).
				\end{align*}
		\item
		    Assuming the finiteness of the Tate-Shafarevich group 
		    of the abelian variety component of $M$, the groups 
			$H^{\ast}_W(S, \mathcal{M}^{\Delta})$ are finitely generated
			abelian groups.
	\end{enumerate}
\end{Cor}

Note that $r_{T} = - \rank Y(K)$ by Corollary \ref{weighted}
and $r_{A} = \rank A(K)$ by \cite{GS20}.

\begin{proof}
	By Proposition \ref{prop: sequence of connected Neron for semiabelian},
	we have a long exact sequence
		\[
				\cdots
			\to
				H_{W}^{i}(S, \mathcal{T}^{0})
			\to
				H_{W}^{i}(S, \mathcal{G}^{0})
			\to
				H_{W}^{i}(S, \mathcal{A}^{0})
			\to
				\cdots
		\]
	up to finite abelian groups.
	If $\Sha(A)$ is finite, then
	the groups $H_{W}^{\ast}(S, \mathcal{A}^{0})$ are finitely generated
	by \cite[Theorem 1.1]{GS20}.
	As $\mathcal{M}^{\Delta} =
	[\mathcal{X}^{\Delta} \to \mathcal{G}^{0}]$,
	the result follows.
\end{proof}

\begin{Prop} \label{prop: decomposition of Euler char of Lie}
We have
	\[
			\chi(S, \Lie \mathcal{G}^{0})
		=
			\chi(S, \Lie \mathcal{T}^{0}) + \chi(S, \Lie \mathcal{A}^{0}).
	\]
\end{Prop}

\begin{proof}
	By the Riemann-Roch formula, it is enough to show that
		\[
			\deg \Lie \mathcal{G}^{0}
				=
			\deg \Lie \mathcal{T}^{0} + \deg \Lie \mathcal{A}^{0}.
		\]
	For this, it is enough to show that
		\[
				\det \Lie \mathcal{G}^{0}
			\cong
				\det \Lie \mathcal{T}^{0} \tensor_{\Order_{S}} \det \Lie \mathcal{A}^{0}
		\]
	as line subbundles of the rank one $K$-vector space
		\[
				\det \Lie G
			\cong
				\det \Lie T \tensor_{K} \det \Lie A.
		\]
	But this is \cite[Theorem 4.1, Question 8.1]{Cha00}.
\end{proof}

\begin{Prop} \label{prop: decomposition of L function}
We have
	\[
			L(M, s)
		=
			L(\mathcal{X}^{\Delta}, s - 1) \cdot
			L(T, s) \cdot L(A, s).
	\]
\end{Prop}

\begin{proof}
	For each place $v$, we have an exact sequence
		\[
				0
			\to
				V_{l}(G)^{I_{v}}
			\to
				V_{l}(M)^{I_{v}}
			\to
				(X \tensor \Q_{l})^{I_{v}}
			\to
				H^{1}(K_{v}^{\ur}, V_{l}(G)).
		\]
	The last term contains the $l$-adic completion of $G(K_{v}^{\ur})$
	tensored with $\Q_{l}$,
	which is $\pi_{0}(\mathcal{G}_{v})(\closure{k(v)}) \tensor \Q_{l}$.
	We have a commutative diagram
		\[
			\begin{CD}
					X(K_{v}^{\ur})
				@>>>
					G(K_{v}^{\ur})
				\\ @VVV @VVV \\
					(X \tensor \Q_{l})^{I_{v}}
				@>>>
					H^{1}(K_{v}^{\ur}, V_{l}(G)).
			\end{CD}
		\]
	Hence the lower horizontal morphism factors through
	$\pi_{0}(\mathcal{G}_{v})(\closure{k(v)}) \tensor \Q_{l}$
	and is equal to the map
	$X(K_{v}^{\ur}) \to \pi_{0}(\mathcal{G}_{v})(\closure{k(v)})$ tensored with $\Q_{l}$.
	The kernel of $X(K_{v}^{\ur}) \to \pi_{0}(\mathcal{G}_{v})(\closure{k(v)})$ is
	$\mathcal{X}_{v}^{\Delta}(\closure{k(v)})$ by definition.
	Therefore we have an exact sequence
		\[
				0
			\to
				V_{l}(G)^{I_{v}}
			\to
				V_{l}(M)^{I_{v}}
			\to
				\mathcal{X}_{v}^{\Delta}(\closure{k(v)}) \tensor \Q_{l}
			\to
				0.
		\]
	This implies
		\[
				L(M, s)
			=
				L(G, s) \cdot L(\mathcal{X}^{\Delta}, s - 1).
		\]
	Proposition \ref{prop: sequence of connected Neron for semiabelian}
	implies that the sequence
		\[
				0
			\to
				V_{l}(T)^{I_{v}}
			\to
				V_{l}(G)^{I_{v}}
			\to
				V_{l}(A)^{I_{v}}
			\to
				0
		\]
	is exact since $V_{l}(G)^{I_{v}} \cong V_{l}(\mathcal{G}^{0}(\Order_{v}^{sh}))$ and so on.
	Hence $L(G, s) = L(T, s) \cdot L(A, s)$.
\end{proof}

\begin{Thm} \label{thm: L value formula for 1 motive}
	Assume that $\Sha(A)$ is finite. 
	Then the groups $H_{W}^{\ast}(S, \mathcal{M}^{\Delta})$ are finitely generated and 
	\[
			\lim_{s \to 1}
				\frac{
					L(M, s)
				}{
					(s - 1)^{r_{M}}
				}
		=
			(-1)^{\rank Y(K)} \cdot
			\chi_{W}(S, \mathcal{M}^{\Delta})^{-1} \cdot
			q^{\chi(S, \Lie \mathcal{G}^{0})} \cdot
			(\log q)^{r_{M}}.
	\]
\end{Thm}

\begin{proof}
	The finiteness of $\Sha(A)$ implies
	the finite generation of $H_{W}^{\ast}(S, \mathcal{A}^{0})$
	by \cite[Theorem 1.1]{GS20}.
	We have
		\begin{align*}
					\chi_{W}(S, \mathcal{M}^{\Delta})
			&	=
					\chi_{W}(S, \mathcal{X}^{\Delta})^{-1} \cdot
					\chi_{W}(S, \mathcal{G}^{0})
			\\
			&	=
					\chi_{W}(S, \mathcal{X}^{\Delta})^{-1} \cdot
					\chi_{W}(S, \mathcal{T}^{0}) \cdot
					\chi_{W}(S, \mathcal{A}^{0})
		\end{align*}
	by Proposition \ref{prop: sequence of connected Neron for semiabelian}.
	With Propositions \ref{prop: decomposition of Q rank},
	\ref{prop: decomposition of Euler char of Lie}
	and \ref{prop: decomposition of L function},
	the theorem reduces to
	Theorem \ref{thm: L value formula for torus} for $T$,
	Theorem \ref{thm: L value formula for Z constructible sheaf} for $\mathcal{X}^{\Delta}$
	and the Weil-\'etale BSD formula \cite[Theorem 1.1]{GS20}
		\[
				\lim_{s \to 1}
					\frac{L(A, s)}{(s - 1)^{\rank A(K)}}
			=
					\chi_{W}(S, \mathcal{A}^{0})^{-1}
				\cdot
					q^{\chi(S, \Lie \mathcal{A}^{0})}
				\cdot
					(\log q)^{\rank A(K)}
		\]
	for $A$.
\end{proof}


\end{document}